\documentclass[microtype,12pt]{article}
\usepackage{custarticle}

\title{Bogomolny Monopoles on \\ Asymptotically Cylindrical $3$-Manifolds}

\author{Wenhan Wang\footnote{\texttt{wenhan-wang@outlook.com}}, \quad
Ryosuke Kimura\footnote{\texttt{kimura.ryosuke.n0@s.mail.nagoya-u.ac.jp}} }

\begin{document}

\maketitle

\begin{abstract}
    We construct irreducible smooth $\SU(2)$-monopoles on asymptotically cylindrical $3$-manifolds with vanishing first and second Betti numbers, using a gluing construction. Their Higgs fields grow linearly along the end, and their Yang-Mills-Higgs energy is infinite. 
    
    We also construct singular $\U(1)$-Dirac monopoles on arbitrary asymptotically cylindrical $3$-manifolds whenever a topological obstruction determined by the singularities and the charges vanishes; this is automatic if the first or second Betti number vanishes.

\end{abstract}

\tableofcontents

\section{Introduction}\label{section: Introduction}

Over the past four decades, the study of the instanton has led to major advances in the topology of $4$-manifolds. In this paper, we study its dimensional reduction, namely the (Bogomolny magnetic) monopole.

\begin{definition}[Monopoles]
    Let $(M,g)$ be an oriented Riemannian $3$-manifold. Let $P \to M$ be a principal $\SU(2)$-bundle, and let $\mathfrak{g}_P=P\times_{\Ad}\su(2)$ be the associated adjoint bundle. Let $A$ be a connection on $P$, and let the Higgs field $\Phi$ be a section of $\mathfrak{g}_P$. The pair $(A,\Phi)$ is called a \emph{(Bogomolny magnetic) monopole} if it satisfies the \emph{Bogomolny equation}
    \begin{equation}
        *F_A=d_A\Phi,
        \label{equation: Bogomolny equation}
    \end{equation}
    where $*$ is the Hodge star determined by the Riemannian metric $g$ and the orientation of $M$.
\end{definition}

A critical point of the \emph{Yang-Mills-Higgs functional}
\begin{equation*}
    \mathcal{YMH}(A,\Phi) = \frac{1}{2} \int_M (\lvert F_A \rvert^2 + \lvert d_A\Phi \rvert^2 ) \vol_g
\end{equation*}
with respect to compactly supported variations is a solution to the \emph{Yang-Mills-Higgs equation}
\begin{equation*}
    d_A^* F_A = [d_A \Phi , \Phi], \qquad d_A^*d_A\Phi=0.
\end{equation*}
By the Bianchi identity $d_AF_A=0$, every monopole is a solution to the Yang-Mills-Higgs equation. On $\R^3$, monopoles of finite energy minimize the Yang-Mills-Higgs functional among configurations with the same mass and charge.

Inspired by the role played by the instanton moduli space in $4$-dimensional topology, we wish to study the monopole moduli space in $3$-dimensional topology. Taubes (\cite{MR614447}) first proved the existence of smooth $\SU(2)$-monopoles with charge $k$ on $\R^3$. Atiyah and Hitchin (\cite{MR934202}) showed that the moduli space of centered $\SU(2)$-monopoles with charge $k$ on $\R^3$ is a $(4k-4)$-dimensional hyperk\"ahler manifold.

However, every smooth monopole on a closed $3$-manifold is necessarily trivial, i.e.\ $*F_A=d_A\Phi=0$. This suggests that we either consider non-compact manifolds (\cite{Oliveira_2016,Foscolo_2016}) or turn to singular monopoles (\cite{esfahani_PhD}). In this paper, we study smooth monopoles on asymptotically cylindrical $3$-manifolds.

\begin{definition}[Asymptotically Cylindrical $3$-Manifolds]
    A complete non-compact connected oriented Riemannian $3$-manifold $(M,g)$ is called \emph{asymptotically cylindrical (ACyl) of rate $\delta>0$}, if there exist a compact set $K \subset M$, a closed Riemannian surface $(\Sigma,g_\Sigma)$, and a diffeomorphism
    \begin{equation*}
        \Upsilon: (1,\infty)_r \times \Sigma \to M\setminus K,
    \end{equation*}
    such that the metric $g_C=(dr)^2+g_\Sigma$ on the cylinder $(1,\infty)_r \times \Sigma$ and its Levi-Civita connection $\nabla_C$ satisfy
    \begin{equation*}
        \lvert\nabla^j_C(\Upsilon^*g-g_C)\rvert_{g_C}= O(e^{-\delta r}), \qquad \text{as } r \to \infty,
    \end{equation*}
    for all $j \in \Z_{\geq 0}$. A smooth function $\rho:M \to \R_{> 0}$ is called a \emph{distance function} if it satisfies $\rho|_{M\setminus K}=r\circ \Upsilon^{-1}$.
    \label{definition: Asymptotically Cylindrical 3-Manifolds}
\end{definition}

We write $\Sigma=\bigsqcup_{j=1}^l\Sigma_j$ with $\Sigma_j$ connected, and say that $M$ has $l=b^0(\Sigma)$ ends. The $j$-th end of $M$ is $U(\Sigma_j):=\Upsilon((1,\infty)\times\Sigma_j)$. Throughout this paper, $\mu_1>0$ denotes the smallest nonzero eigenvalue of the Laplacian on $(\Sigma,g_\Sigma)$.

We first prove the existence of singular $\U(1)$-Dirac monopoles in Subsections~\ref{subsection: Local Models of U(1)-Dirac Monopoles} and \ref{subsection: Construction of SU(2)-Dirac Monopoles}.
\begin{theorem}[Existence of $\U(1)$-Dirac Monopoles]
    Let $(M,g)$ be an asymptotically cylindrical $3$-manifold with $l$ ends and a distance function $\rho$, and let $p_1,\cdots,p_s\in M$ be distinct points. Given $\overline{m}\in\R$, $\vec{k}^I=(k_1^I, \cdots, k_s^I) \in \Z^s$ and $\vec{k}^O=(k_1^O, \cdots, k_l^O) \in \Z^l$ satisfying $\sum\limits_{i=1}^{s} k_i^I = \sum\limits_{j=1}^l k_j^O$, there is a unique harmonic function $\phi_D$ on $M\setminus\{p_1,\cdots,p_s\}$ with the following asymptotic expansions near the singularities and along the ends:
    \begin{align*}
        \phi_D &= m_i^I - \frac{k_i^I}{ r_i} + O ( r_i ), \qquad \text{as }r_i \to 0, \\
        \phi_D &= m_j^O +\frac{4\pi k_j^O}{\vol(\Sigma_j)} \rho + O ( e^{-\delta'\rho} ), \qquad \text{on }U(\Sigma_j),\text{ as }\rho \to \infty,
    \end{align*}
    where $r_i = \operatorname{dist} ( \cdot, p_i )$, $\delta'>0$ is a constant, and the mass $\vec{m}^O=(m_1^O,\cdots,m_l^O)$ satisfies $\frac{1}{l}(m_1^O+\cdots+m_l^O)=\overline{m}$. The mass $\vec{m}^O$ is determined by $(M,g)$, the choice of $\rho$, the singularities, the charges and $\overline{m}$, and in general it cannot be prescribed (Remark~\ref{remark: masses are determined}); if $M$ has only one end, then $m_1^O=\overline{m}$.

    Moreover, if the class $b(p_1,\cdots,p_s;\vec k^I,\vec k^O)\in H^2(M;\R/4\pi\Z)$ of Theorem~\ref{theorem:well definedness of b(p,[mO])} vanishes, where $\U(1)$ is identified with $\R/4\pi\Z$ as in Subsection~\ref{subsection: Construction of SU(2)-Dirac Monopoles}, then there exists a connection $a_D$ on a principal $\U(1)$-bundle $L' \to M \setminus \{p_1,\cdots,p_s\}$ such that $(a_D,\phi_D)$ is a singular $\U(1)$-Dirac monopole with charges $k_1^I,\cdots,k_s^I$ (Definition~\ref{definition: U(1)-Dirac monopoles}). The pair $(L',a_D)$ is unique up to gauge equivalence and tensoring with flat $\U(1)$-bundles, which are parametrized by $H^1(M;\R/4\pi\Z)$. The class $b$ does not depend on $\overline{m}$, and it vanishes for all choices of the singularities and the charges if $b^2(M)=l-1$, in particular if $b^1(M)=0$ or $b^2(M)=0$.
\end{theorem}

\begin{remark}
    In the case of asymptotically conical $3$-manifolds, the Higgs field $\phi_D$ always exists, and the only condition for the existence of the $\U(1)$-Dirac monopole is the vanishing of the class of $*d\phi_D$ in $H^2(M\setminus\{p_1,\cdots,p_s\};\R/4\pi\Z)$ (\cite{Oliveira_2016}). However, the roles of the masses and the charges of the ends are interchanged: in \cite{Oliveira_2016}, the masses are prescribed and the charges of the ends are real numbers determined by them, whereas here the charges $\vec k^O$ are prescribed integers and the masses are determined by the other data. The condition $b^2(M)=l-1$ means that $H_2(M;\R)$ is spanned by the cross-sections $\Sigma_1,\cdots,\Sigma_l$ of the ends; in general, $b^2(M)\geq l-1$ (Lemma~\ref{lemma: topology of the ends}). If $b^2(M)>l-1$ and some singularity has nonzero charge, then $b\neq0$ for suitable positions of the singularities (Remark~\ref{remark: obstruction is sharp}), so the singularities have to be chosen in the zero set of $b$. This zero set is nonempty as soon as at least $b^2(M)-l+1$ singularities have nonzero charges (Proposition~\ref{proposition: admissible singularities}). For example, $M=\R\times T^2$ satisfies $b^1(M)=2$ and $b^2(M)=1=l-1$, while $M=(S^1\times S^2)\setminus\overline{B}$ with one end $S^2$ satisfies $b^2(M)=1>l-1=0$. On $\R\times T^2$ with the flat metric, monopoles whose Higgs fields grow at most linearly along both ends are known as monopole walls (\cite{Ward_2007,Cherkis_Ward_2012}).
\end{remark}

In order to construct the approximate solution $(A_{\mathrm{app}},\Phi_{\mathrm{app}})$, we assume $s=n$ and $k_1^I=\cdots=k_{n}^I=1$. We prove the existence of irreducible smooth $\SU(2)$-monopoles in the rest of this paper.
\begin{theorem}[Existence of $\SU(2)$-Monopoles]
    Let $(M,g)$ be an asymptotically cylindrical $3$-manifold of rate $\delta$ with $b^1(M)=b^2(M)=0$ and a distance function $\rho$. Given $n\geq1$ distinct points $p_1,\cdots,p_n \in M$ and $\beta \in (\beta^*,0)$ with $\beta^*=\max\{-\delta,-\sqrt{\mu_1}\}$, for sufficiently large $m \in \R_{> 0}$, there exists, on a principal $\SU(2)$-bundle $P \to M$, an irreducible smooth $\SU(2)$-monopole with charge $n$ and mass $m$ in the sense of Definition~\ref{definition: mass and charge ACyl}. It is gauge equivalent to $(A_{\mathrm{app}},\Phi_{\mathrm{app}}) + (a,\varphi)$ by a gauge transformation of class $L^2_{2,loc}$, where the perturbation $(a,\varphi)$ is small in the weighted Sobolev space $\mathcal{H}^1_\beta$ of Definition~\ref{definition: weighted Sobolev spaces}:
    \begin{equation*}
        \Vert(a,\varphi)\Vert_{\mathcal{H}^1_{\beta}}\le C m^{-\frac54},
    \end{equation*}
    where $C >0$ is a constant independent of $m$.
    \label{theorem: main theorem}
\end{theorem}

\begin{remark}\label{remark: one end}
    Since $b^1(M)=b^2(M)=0$, there is only one end of $M$, and $\Sigma\cong S^2$ (Lemma~\ref{lemma: topology of the ends}). The mass of the monopole is $m=\overline{m}$, and its Higgs field $\Phi$ satisfies $\lvert\Phi\rvert=m+\frac{4\pi n}{\vol(\Sigma)}\rho+o(1)$ along the end. In particular, $\lvert\Phi\rvert\to\infty$ along the end, so $\lim\lvert\Phi\rvert$, which defines the mass of a monopole on $\R^3$, is infinite, and our monopoles have infinite energy.
\end{remark}

Our strategy is a gluing method, and the organization of this paper follows it.
\begin{itemize}
    \item In Section~\ref{section: Dirac Monopoles}, we construct a scaled $\SU(2)$-Dirac monopole $(A_D^{\overline{m}},\Phi_D^{\overline{m}})$ with singularities $\{p_1,\cdots,p_{n}\}$ by lifting a $\U(1)$-Dirac monopole, and show that its Higgs field can be made sufficiently large away from the singularities by increasing the average mass $\overline{m}$.

    \item In Section~\ref{section: Approximate Solutions}, we construct an approximate solution $(A_{\mathrm{app}},\Phi_{\mathrm{app}})$ by gluing pull-back scaled PS monopoles $(\eta^*_i(A_{PS}^{\lambda_i}),\eta^*_i(\Phi_{PS}^{\lambda_i}))$ onto these singularities, and bound the error term $e_0$ independently of the average mass.

    \item In Section~\ref{section: Genuine Solutions}, we add a small perturbation $(a,\varphi)$ to the approximate solution to deform it into a genuine solution. In suitable weighted Sobolev spaces, we construct a right inverse of the linearized operator $d_2$ whose norm is bounded independently of $\overline m$, and then apply the contraction mapping theorem.
\end{itemize}

The main difficulties are threefold: the linear growth of the Higgs field along the cylindrical end in Section~\ref{section: Dirac Monopoles}, the choice of the scaling factors of the PS monopoles in Section~\ref{section: Approximate Solutions}, and the choice of weighted Sobolev spaces and cut-off functions adapted both to the singularities and to the end in Section~\ref{section: Genuine Solutions}. We expect that this approach can be extended to the cases $b^1(M)\neq 0$ or $b^2(M) \neq 0$.

\begin{acknowledgements}
    The authors thank Masashi Hamanaka for his interest in our work and Weifeng Sun for helpful discussions. The first author is most grateful to his father, Junhua Wang, and his mother, Huizhen Xu, for their constant selfless love.
\end{acknowledgements}

\section{Preliminaries}\label{section: Preliminaries}

\subsection{Deformation Complexes}\label{subsection: Deformation Complexes}

In this subsection, we introduce the deformation complex associated with the Bogomolny equation.

For our later purpose, to find a genuine solution to the Bogomolny equation near the approximate solution $(A_{\mathrm{app}},\Phi_{\mathrm{app}})$, we look for a ``small'' (in the sense of a suitable norm) perturbation $(a,\varphi)$, such that $(A_{\mathrm{app}}+a,\Phi_{\mathrm{app}}+\varphi)=(A_{\mathrm{app}},\Phi_{\mathrm{app}})+(a,\varphi)$ is a genuine solution.

The condition that $(A_{\mathrm{app}}+a,\Phi_{\mathrm{app}}+\varphi)$ is a genuine solution can be transformed into an equation for $(a,\varphi)$:
\begin{align}
        & & *F_{A_{\mathrm{app}}+a}-d_{A_{\mathrm{app}}+a}(\Phi_{\mathrm{app}}+\varphi) &=0 \notag \\
        &\implies &\underbrace{(*d_{A_{\mathrm{app}}}a-d_{A_{\mathrm{app}}}\varphi-[a,\Phi_{\mathrm{app}}])}_{\text{linear term}}+ \underbrace{(*\frac{1}{2}[a\wedge a]-[a,\varphi])}_{\text{quadratic term}}+\underbrace{(*F_{A_{\mathrm{app}}}-d_{A_{\mathrm{app}}}\Phi_{\mathrm{app}})}_{\text{error term}} &=0.
        \label{equation: three terms}
\end{align}

The equation~\eqref{equation: three terms} can be divided into three parts: the linear term, the quadratic term, and the error term. We discuss these three parts in turn.
\begin{itemize}
    \item The first part is the linear term. Let
    \begin{align*}
        d_2 := d_2^{(A_{\mathrm{app}},\Phi_{\mathrm{app}})}:\Omega^1(M;\mathfrak{g}_P) \oplus \Omega^0(M;\mathfrak{g}_P)  &\to \Omega^1(M;\mathfrak{g}_P) \\
        (a,\varphi) &\mapsto *d_{A_{\mathrm{app}}}a-d_{A_{\mathrm{app}}}\varphi-[a,\Phi_{\mathrm{app}}]
    \end{align*}
    be the \emph{linearized operator} of the Bogomolny equation at $(A_{\mathrm{app}},\Phi_{\mathrm{app}})$. The formal adjoint of $d_2$ with respect to the $L^2$-inner product is
    \begin{align}
        d_2^*: \Omega^1(M;\mathfrak{g}_P) &\to \Omega^1(M;\mathfrak{g}_P) \oplus \Omega^0(M;\mathfrak{g}_P) \notag \\
        u &\mapsto (*d_{A_{\mathrm{app}}}u+[u,\Phi_{\mathrm{app}}] , -d_{A_{\mathrm{app}}}^*u) \label{equation:d2*u=(*dA0u+[u,Phi0],-dA0*u)}
    \end{align}

    \item The second part is the quadratic term. Let
    \begin{align*}
        Q := Q^{(A_{\mathrm{app}},\Phi_{\mathrm{app}})}:(\Omega^1(M;\mathfrak{g}_P) \oplus \Omega^0(M;\mathfrak{g}_P))^{2} &\to \Omega^1(M;\mathfrak{g}_P) \\
        ((a,\varphi),(a',\varphi')) &\mapsto *\frac{1}{2}[a\wedge a']-\frac{1}{2}([a,\varphi']+[a',\varphi])
    \end{align*}
    be the \emph{quadratic operator} of the Bogomolny equation at $(A_{\mathrm{app}},\Phi_{\mathrm{app}})$. It is symmetric and bilinear, since $[a\wedge a']=[a'\wedge a]$, and we also use the notation $Q(a,\varphi):=Q((a,\varphi),(a,\varphi))=*\frac{1}{2}[a\wedge a]-[a,\varphi]$.

    \item The third part is the error term. Let
    \begin{align}
        e_0:\mathcal{A}(P) \times \Omega^0(M;\mathfrak{g}_P)  &\to \Omega^1(M;\mathfrak{g}_P) \notag \\
        (A_{\mathrm{app}},\Phi_{\mathrm{app}}) &\mapsto *F_{A_{\mathrm{app}}}-d_{A_{\mathrm{app}}}\Phi_{\mathrm{app}}
        \label{equation: error terms}
    \end{align}
    be the \emph{error operator}, measuring the difference between the approximate solution and the genuine solution, where $\mathcal{A}(P)$ is the space of connections on $P$.
\end{itemize}

Using the notation above, the Bogomolny equation~\eqref{equation: three terms} can be written as
\begin{equation}
    d_2(a,\varphi)+Q((a,\varphi),(a,\varphi))+e_0 = 0.
    \label{equation: not elliptic}
\end{equation}
However, this equation is not elliptic since it is gauge invariant under the gauge group of $P$.

The \emph{infinitesimal gauge action} at $(A_{\mathrm{app}},\Phi_{\mathrm{app}})$ is
\begin{align*}
    d_1 := d_1^{(A_{\mathrm{app}},\Phi_{\mathrm{app}})}:\Omega^0(M;\mathfrak{g}_P) &\to \Omega^1(M;\mathfrak{g}_P) \oplus \Omega^0(M;\mathfrak{g}_P) \\
    \xi &\mapsto (-d_{A_{\mathrm{app}}}\xi,-[\Phi_{\mathrm{app}},\xi]).
\end{align*}
The formal adjoint of $d_1$ with respect to the $L^2$-inner product is
\begin{align*}
    d_1^*: \Omega^1(M;\mathfrak{g}_P) \oplus \Omega^0(M;\mathfrak{g}_P) &\to \Omega^0(M;\mathfrak{g}_P) \\
    (a,\varphi) &\mapsto -d_{A_{\mathrm{app}}}^*a+[\Phi_{\mathrm{app}},\varphi].
\end{align*}

In the remainder of this subsection, we introduce the \emph{deformation complex}
\begin{equation}
    \Omega^0(M;\mathfrak{g}_P) \xrightarrow{d_1^{(A,\Phi)}} \Omega^1(M;\mathfrak{g}_P) \oplus \Omega^0(M;\mathfrak{g}_P) \xrightarrow{d_2^{(A,\Phi)}} \Omega^1(M;\mathfrak{g}_P)
    \label{equation: deformation complex}
\end{equation}
associated to the Bogomolny equation. If $(A,\Phi)$ is a monopole, then $d_2 \circ d_1=0$, and the sequence~\eqref{equation: deformation complex} is an elliptic complex. The elliptic operator
\begin{equation*}
    D:= D^{(A,\Phi)} = d_2 \oplus d_1^* :  \Omega^1(M;\mathfrak{g}_P) \oplus \Omega^0(M;\mathfrak{g}_P) \to \Omega^1(M;\mathfrak{g}_P) \oplus \Omega^0(M;\mathfrak{g}_P)
\end{equation*}
defines the \emph{gauge-fixed elliptic Bogomolny equation}
\begin{equation*}
    D(a,\varphi) = (f,0),
\end{equation*}
where $d_1^*(a,\varphi)=0$ describes the \emph{local slice} for the gauge action at $(A,\Phi)$.

As a special version of the Weitzenb\"ock formula for connections on vector bundles, we have the following Weitzenb\"ock formula in the context of the Bogomolny equation.

\begin{lemma}[Monopole Weitzenb\"ock Formulas]\label{lemma: Weitzenbock D}
    Let $(M,g)$ be an oriented Riemannian $3$-manifold, and let $(A,\Phi)$ be a configuration on $M$. Set $e:=*F_A-d_A\Phi$. The \emph{monopole Weitzenb\"ock formula} is
    \begin{equation}
        d_2d_2^* u =\nabla^*_{A}\nabla_{A}u -[\Phi,[\Phi,u]] + \Ric(u) +*[e\wedge u], \label{equation:d2d2*=Weitzenbock}
    \end{equation}
    where $u \in \Omega^1(M;\mathfrak{g}_P)$, and $d_1$, $d_2$, and $D$ are the operators associated to $(A,\Phi)$.

    For every $\zeta \in \Omega^0(M;\mathfrak{g}_P)$, we have $d_2d_1\zeta=-[e,\zeta]$. Hence, the section $e\star u:=d_1^*d_2^*u=(d_2d_1)^*u$ is algebraic in $u$ and satisfies $\lvert e\star u\rvert\le C\lvert e\rvert\lvert u\rvert$. Moreover, for every compactly supported pair $\vartheta$ of class $L^2_1$, we have
    \begin{equation}\label{equation: Weitzenbock D}
        \Vert D\vartheta\Vert^2_{L^2}=\Vert\nabla_{A}\vartheta\Vert^2_{L^2}+\Vert[\Phi,\vartheta]\Vert^2_{L^2}+\langle\mathcal Z\vartheta,\vartheta\rangle_{L^2},
    \end{equation}
    where $\nabla_A$ and $\ad(\Phi)$ act componentwise on $\vartheta$, and $\mathcal Z$ is a self-adjoint algebraic endomorphism satisfying
    \begin{equation*}
        \lvert\mathcal Z\rvert\le C(\lvert\Ric_g\rvert+\lvert F_A\rvert+\lvert d_A\Phi\rvert).
    \end{equation*}
\end{lemma}
\begin{proof}
    By \eqref{equation:d2*u=(*dA0u+[u,Phi0],-dA0*u)} with $(A,\Phi)$ in place of $(A_{\mathrm{app}},\Phi_{\mathrm{app}})$, the first-order terms $\pm[*d_Au,\Phi]$ cancel, and
    \begin{equation*}
        d_2d_2^* u =(d_A^*d_A+d_Ad_A^*)u -[\Phi,[\Phi,u]] -*[d_A\Phi\wedge u].
    \end{equation*}
    Together with $(d_A^*d_A+d_Ad_A^*)u=\nabla^*_{A}\nabla_{A}u+\Ric(u)+*[*F_A\wedge u]$, this gives \eqref{equation:d2d2*=Weitzenbock}.

    For $\zeta\in\Omega^0(M;\mathfrak g_P)$, the Leibniz rule, together with
    $d_A^2\zeta=[F_A,\zeta]$ and
    $d_A[\Phi,\zeta]+[d_A\zeta,\Phi]=[d_A\Phi,\zeta]$, gives
    \begin{equation*}
        \begin{aligned}
            d_2d_1\zeta
            &=-*d_A^2\zeta+d_A[\Phi,\zeta]+[d_A\zeta,\Phi]\\
            &=-[*F_A,\zeta]+[d_A\Phi,\zeta]\\
            &=-[e,\zeta].
        \end{aligned}
    \end{equation*}
    Taking formal adjoints, we obtain $e\star u:=d_1^*d_2^*u=(d_2d_1)^*u$. Since $d_2d_1$ is algebraic in $\zeta$, its formal adjoint is algebraic in $u$ and satisfies the stated pointwise bound.

    To prove the identity for $D$, write $D\vartheta=\mathcal D\vartheta+\ad(\Phi)\vartheta$, where
    $\mathcal D(a,\varphi)=(*d_Aa-d_A\varphi,-d_A^*a)$.
    The operator $\mathcal D$ is formally self-adjoint, since $*d_A$ is formally self-adjoint and
    $-d_A$ and $-d_A^*$ are formal adjoints. The operator $\ad(\Phi)$ is skew-adjoint. Thus
    \begin{equation*}
        \Vert D\vartheta\Vert_{L^2}^2
        =\Vert\mathcal D\vartheta\Vert_{L^2}^2+\Vert[\Phi,\vartheta]\Vert_{L^2}^2
        +\langle[\mathcal D,\ad(\Phi)]\vartheta,\vartheta\rangle_{L^2}.
    \end{equation*}
    The commutator is a zeroth-order operator whose coefficients are linear in $d_A\Phi$.
    Define the parallel isometries
    \begin{equation*}
        \begin{aligned}
            T:\Omega^1(M;\mathfrak g_P)\oplus\Omega^0(M;\mathfrak g_P)
            &\to\Omega^1(M;\mathfrak g_P)\oplus\Omega^3(M;\mathfrak g_P),\\
            (a,\varphi)&\mapsto a+\varphi\vol_g,\\
            S:\Omega^0(M;\mathfrak g_P)\oplus\Omega^2(M;\mathfrak g_P)
            &\to\Omega^1(M;\mathfrak g_P)\oplus\Omega^0(M;\mathfrak g_P),\\
            (\omega_0,\omega_2)&\mapsto(*\omega_2,-\omega_0).
        \end{aligned}
    \end{equation*}
    Then $\mathcal D=S\circ(d_A+d_A^*)\circ T$. The Weitzenb\"ock formula for $d_A+d_A^*$ gives
    \begin{equation*}
        \Vert\mathcal D\vartheta\Vert_{L^2}^2
        =\Vert\nabla_A\vartheta\Vert_{L^2}^2
        +\langle\mathcal R\vartheta,\vartheta\rangle_{L^2},
    \end{equation*}
    where, after transferring the curvature term to $\Omega^1\oplus\Omega^0$, $\mathcal R$ is algebraic and satisfies $\lvert\mathcal R\rvert\le C(\lvert\Ric_g\rvert+\lvert F_A\rvert)$. Therefore, setting $\mathcal Z:=\mathcal R+[\mathcal D,\ad(\Phi)]$ gives the stated identity and pointwise bound.
\end{proof}

The Bogomolny equation inherits a scaling property from the conformal invariance of the ASD equation (\cite{Oliveira_2016}). This yields a one-to-one correspondence between monopoles on $(M,g)$ and those on $(M,\lambda^2 g)$.
\begin{proposition}[Scaling Invariance]
    Let $(M,g)$ be an oriented Riemannian $3$-manifold, and let $(A,\Phi)$ be a monopole on $M$. Then for every constant $\lambda>0$, $(A,\lambda^{-1}\Phi)$ is a monopole on $(M, \lambda^2 g)$.
    \label{proposition: Scaling Invariance}
\end{proposition}
\begin{proof}
    The transformation $*_{\lambda^2 g} = \lambda^{n-2k}*_g$ for $n$-dimensional Hodge star operators acting on $k$-forms induces
    \begin{equation*}
        *_{\lambda^2 g}F_A = \lambda^{3-2\times 2}*_g F_A
        = \lambda^{-1} (*_g F_A)
        = \lambda^{-1} (d_A \Phi)
        =d_A (\lambda^{-1}\Phi).
    \end{equation*}
\end{proof}

\subsection{Analysis on Asymptotically Cylindrical $3$-Manifolds}\label{subsection: Analysis on Asymptotically Cylindrical $3$-Manifolds}

In this subsection, we introduce the Lockhart-McOwen theory (\cite{MR837256}) for elliptic operators on the asymptotically cylindrical $3$-manifold $(M,g)$, and apply it to the Laplacian $\Delta=d^*d$ on functions.

Let $E\to M$ be a vector bundle equipped with a fiber metric $h_E$ and a connection $\nabla_E$ compatible with $h_E$. The vector bundle $E$ is called \emph{asymptotically cylindrical (ACyl)} if there is a vector bundle $E_0\to(1,\infty)\times\Sigma$ with a fiber metric $h_{E_0}$ and a compatible connection $\nabla_{E_0}$, which are invariant under translations in the $r$-direction, such that $\Upsilon^*(E)\cong E_0$, and under this isomorphism, $\nabla^j_{E_0}(\Upsilon^* (h_E ) - h_{E_0})$ and $\nabla^j_{E_0}(\Upsilon^* (\nabla_E ) - \nabla_{E_0})$ are $O (e^{- \delta r} )$ as $r \to \infty$, for all $j \in \Z_{\geq 0}$. For instance, the trivial line bundle and the bundles of differential forms on $M$ are asymptotically cylindrical.

\begin{definition}[Lockhart-McOwen Sobolev Spaces]\label{definition: Lockhart-McOwen Sobolev spaces}
    Let $\rho$ be a distance function. For $p \geq 1$, $k \geq 0$, and $\beta \in \R$, we define the norm
    \begin{align*}
        \Vert u \Vert_{L^p_{k, \beta}}^p = \sum_{i=0}^k \int_M  e^{- \beta \rho}\lvert \nabla_E^i u \rvert_{h_E}^p \vol_g
    \end{align*}
    for all sections $u \in \Gamma (E )$. The \emph{Lockhart-McOwen Sobolev space} $L^p_{k, \beta} (E )$ is the Banach space of all $k$-times weakly differentiable sections $u$ with $\Vert u \Vert_{L^p_{k, \beta}} < \infty$.
\end{definition}

Roughly speaking, the sections in $L^p_{k,\beta}(E)$ decay faster than $e^{\frac{\beta}{p}\rho}$ along the ends, and we call $\frac{\beta}{p}$ the \emph{rate} of $L^p_{k,\beta}(E)$. Let $E,E'\to M$ be asymptotically cylindrical vector bundles. A linear partial differential operator $F : \Gamma (E ) \to \Gamma (E' )$ of order $\kappa$ is called \emph{asymptotically cylindrical} if, under $\Upsilon$, the coefficients of $F$ converge, together with all their derivatives, at the rate $O(e^{-\delta r})$ to those of an operator $F_0:\Gamma(E_0)\to\Gamma(E_0')$ which is invariant under translations in the $r$-direction. It is standard that $F$ then extends to bounded linear operators
\begin{align*}
    F^p_{k+\kappa, \beta} : L^p_{k+\kappa, \beta} (E ) \to L^p_{k, \beta} (E').
\end{align*}

\begin{definition}\label{definition: Fredholm}
    Suppose that $F$ is elliptic. Denote by $E_0\otimes \mathbb C$ the complexification, and extend $F_0$ to a complex linear operator $F_0 : \Gamma(E_0 \otimes \mathbb C) \to \Gamma(E'_0 \otimes \mathbb C)$. A real number $\gamma$ is called a \emph{critical rate} of $F$ if there are $\nu\in\R$ and a nonzero section $u$ of $E_0 \otimes \mathbb C$, which is invariant under translations in the $r$-direction, such that $F_0(e^{(\gamma + i\nu)r} u) = 0$. We denote by $\mathcal{D}_{F_0}\subset\R$ the set of critical rates. For $\gamma\in\mathcal{D}_{F_0}$, we denote by $d(\gamma)$ the dimension of the space of all solutions of $F_0u=0$ which are finite sums of sections of the form $e^{(\gamma + i\nu)r}r^ju_{\nu,j}$, where $\nu\in\R$, $j\in\Z_{\geq 0}$, and $u_{\nu,j}$ is a section of $E_0 \otimes \mathbb C$ invariant under translations in the $r$-direction.
\end{definition}

\begin{theorem}[{\cite{MR837256}}]
    Suppose that $F$ is elliptic. The set $\mathcal{D}_{F_0} \subset \R$ is discrete. For $1<p<\infty$, $k \geq 0$, and $\beta \in \R$, the bounded linear operator $F^p_{k+\kappa, \beta} : L^p_{k+\kappa, \beta} (E ) \to L^p_{k, \beta} (E' )$ is Fredholm if and only if $\frac{\beta}{p} \notin \mathcal{D}_{F_0}$. Moreover, if $\beta_1<\beta_2$ and $\frac{\beta_1}{p},\frac{\beta_2}{p}\notin\mathcal{D}_{F_0}$, then the indices satisfy
    \begin{equation*}
        \operatorname{ind} F^p_{k+\kappa, \beta_2}-\operatorname{ind} F^p_{k+\kappa, \beta_1}=\sum_{\gamma}d(\gamma),
    \end{equation*}
    where the sum is taken over all critical rates $\gamma$ with $\frac{\beta_1}{p}<\gamma<\frac{\beta_2}{p}$.
    \label{theorem: Fredholm operator}
\end{theorem}

For example, the translation invariant model of $d^*d$ on functions is $-\partial_r^2+\Delta_\Sigma$. Separating variables, we see that $\mathcal{D}_{(d^*d)_0}$ consists of $0$ and $\pm\sqrt{\mu}$, where $\mu$ runs over the nonzero eigenvalues of the Laplacian on $(\Sigma,g_\Sigma)$, and that the solutions for $\gamma=0$ are the functions which are equal to $a_j+b_jr$ on $(1,\infty)\times\Sigma_j$ for some constants $a_j,b_j$. Hence $d(0)=2l$. In this paper, we only need the case $p=2$.

\begin{lemma}\label{lemma: d*d(L(M)) = {intM u vol=0}}
    Let $\mathcal{K}$ be the space of harmonic functions $h$ on $M$ of at most linear growth, i.e.\ $\lvert h \rvert \leq C(1+\rho)$ for some constant $C>0$. Let $\beta\in(-\min\{\delta,\sqrt{\mu_1}\},0)$ and $k\geq0$. Then the following hold.
    \begin{enumerate}[label=\textbf{(\arabic*)},listparindent=\parindent]
        \item The operator $d^* d:L^2_{k+2, \beta} (M ) \to L^2_{k, \beta} (M )$ is injective, and its image is the closed subspace
        \begin{align*}
            \{ u \in L^2_{k, \beta} (M ) \mid \int_M u h \vol_g =0 \text{ for all } h\in\mathcal{K} \},
        \end{align*}
        on which it has a bounded inverse.
        \item The space $\mathcal{K}$ has dimension $l$. Every $h\in\mathcal{K}$ has the asymptotic expansion
        \begin{equation}
            h=a_j+b_j\rho+O(e^{-\delta'\rho})\qquad\text{on }U(\Sigma_j)
            \label{equation: expansion of harmonic functions of linear growth}
        \end{equation}
        for some $a_j,b_j\in\R$ and $\delta'>0$, where the derivatives of the $O(e^{-\delta'\rho})$-term are also $O(e^{-\delta'\rho})$, and $\sum_j b_j\vol(\Sigma_j)=0$. If $b_j=0$ for all $j$, then $h$ is constant. In particular, if $M$ has only one end, then $\mathcal{K}=\R$.
    \end{enumerate}
\end{lemma}

\begin{proof}
    The rates $\frac{\beta}{2}$ and $-\frac{\beta}{2}$ of the weighted spaces below lie in $(-\min\{\delta,\sqrt{\mu_1}\},\min\{\delta,\sqrt{\mu_1}\})\setminus\{0\}$; in particular, they are not critical rates of $d^*d$. By Theorem~\ref{theorem: Fredholm operator}, $d^*d:L^2_{k+2,\pm\beta}(M)\to L^2_{k,\pm\beta}(M)$ are Fredholm, so their images are closed.

    We first show that $d^*d$ is injective on $L^2_{k+2,\beta}(M)$. Let $u \in L^2_{k+2, \beta} (M)$ satisfy $d^*du = 0$. Since $\beta < 0$, we have $u, du \in L^2(M)$, so the following integration by parts is justified. Let $\chi_R$ be a smooth cut-off function with $\chi_R \equiv 1$ on $\{\rho \le R\}$, $\chi_R \equiv 0$ on $\{\rho \ge R+1\}$, and $\lvert d\chi_R \rvert \le C$. Then
    \begin{equation*}
        0 = \int_M \langle \chi_R u, d^*du  \rangle\vol_g
        = \int_M \chi_R \lvert du\rvert^2 \vol_g
        + \int_M u  \langle d\chi_R, du \rangle \vol_g.
    \end{equation*}
    By the Cauchy-Schwarz inequality, the last term is bounded by $C\Vert u\Vert_{L^2(\{R\le\rho\le R+1\})}\Vert du\Vert_{L^2}$, which tends to $0$ as $R\to\infty$. Hence $du = 0$, and $u\equiv c$ is constant, since $M$ is connected. If $c\neq0$, then $\Vert c\Vert_{L^2_{0,\beta}}^2=c^2\int_Me^{-\beta\rho}\vol_g=\infty$, since $\beta<0$, which is a contradiction. Thus $u=0$. Since the image of $d^*d$ is closed, the open mapping theorem gives a bounded inverse on the image.

    To describe the image, we compute the cokernel. For $k\geq0$, let $L^2_{-k,-\beta}(M)$ be the dual space of $L^2_{k,\beta}(M)$ with respect to the $L^2$-pairing; for $k=0$, it is the space of Definition~\ref{definition: Lockhart-McOwen Sobolev spaces} with $(p,\beta)$ replaced by $(2,-\beta)$. The cokernel of $d^*d$ on $L^2_{k+2,\beta}(M)$ is the kernel of $d^*d$ on $L^2_{-k,-\beta}(M)$, which by elliptic regularity is the space $\mathcal{K}_\beta$ of harmonic functions in $L^2_{0,-\beta}(M)$; moreover, $\mathcal{K}_\beta\subset L^2_{k',-\beta}(M)$ for all $k'\geq0$. Similarly, the cokernel of $d^*d$ on $L^2_{k+2,-\beta}(M)$ is the space of harmonic functions in $L^2_{0,\beta}(M)$, which vanishes by the injectivity above. Hence the indices of $d^*d$ on $L^2_{k+2,\beta}(M)$ and on $L^2_{k+2,-\beta}(M)$ are $-\dim\mathcal{K}_\beta$ and $\dim\mathcal{K}_\beta$, respectively. Since $0$ is the only critical rate between $\frac{\beta}{2}$ and $-\frac{\beta}{2}$, and $d(0)=2l$, Theorem~\ref{theorem: Fredholm operator} gives $\dim\mathcal{K}_\beta=l$.

    We next show that $\mathcal{K}_\beta=\mathcal{K}$. For $j=1,\cdots,l$, let $\chi_j:M\to[0,1]$ be a smooth function with $\chi_j=1$ on $\Upsilon([3,\infty)\times\Sigma_j)$ and $\supp\chi_j\subset\Upsilon([2,\infty)\times\Sigma_j)$, and let $W$ be the space spanned by the $2l$ functions $\chi_j$ and $\chi_j\rho$. Since the metric is asymptotically cylindrical, $d^*d(\chi_j\rho)$ and all its derivatives are $O(e^{-\delta\rho})$, so $d^*d$ maps $W\oplus L^2_{k+2,\beta}(M)$ to $L^2_{k,\beta}(M)$, since $-\beta<\delta$. This operator has index $2l-\dim\mathcal{K}_\beta=l$, so its kernel has dimension at least $l$. Since no nonzero element of $W$ lies in $L^2_{0,\beta}(M)$, taking the sum of the two components embeds this kernel into $\mathcal{K}_\beta$, which has dimension $l$. Hence every $h\in\mathcal{K}_\beta$ can be written uniquely as
    \begin{equation*}
        h=\sum_{j=1}^l(a_j+b_j\rho)\chi_j+v,\qquad a_j,b_j\in\R,\quad v\in L^2_{k+2,\beta}(M).
    \end{equation*}
    By uniqueness, $v$ does not depend on $k$, so $v\in L^2_{k,\beta}(M)$ for all $k$, and the Sobolev embedding on balls of unit size shows that $v$ and all its derivatives are $O(e^{\frac12\beta\rho})$. This gives the expansion \eqref{equation: expansion of harmonic functions of linear growth} with $\delta'=-\frac12\beta$, and in particular $\mathcal{K}_\beta\subset\mathcal{K}$. Conversely, $\mathcal{K}\subset\mathcal{K}_\beta$, since $\beta<0$. Thus $\mathcal{K}=\mathcal{K}_\beta$ has dimension $l$, and an element $u\in L^2_{k,\beta}(M)$ lies in the image of $d^*d$ if and only if $\int_M u h \vol_g = 0$ for all $h\in\mathcal{K}$. This proves (1).

    For (2), it remains to consider the coefficients $b_j$. Since $\Delta h=0$, Stokes' theorem on $\{\rho\leq R\}$ gives
    \begin{equation*}
        0=\int_{\{\rho\leq R\}}\Delta h\vol_g=-\sum_j\int_{\{\rho=R\}\cap U(\Sigma_j)}\langle dh,\nu\rangle\longrightarrow-\sum_j b_j\vol(\Sigma_j),\qquad\text{as }R\to\infty,
    \end{equation*}
    where $\nu$ is the outward unit normal. If $b_j=0$ for all $j$, then $h$ is bounded and $\lvert dh\rvert=O(e^{-\delta'\rho})$, and the integration by parts in the proof of (1) gives $dh=0$. If $M$ has only one end, then $b_1\vol(\Sigma_1)=0$, so $\mathcal{K}=\R$.
\end{proof}

We end this subsection with the topology of the ends. Let $\overline{M}:=M\setminus\Upsilon((2,\infty)\times\Sigma)$. It is a compact $3$-manifold with boundary $\partial\overline{M}\cong\Sigma$, and $M$ is homotopy equivalent to it.

\begin{lemma}\label{lemma: topology of the ends}
    \begin{enumerate}[label=\textbf{(\arabic*)},listparindent=\parindent]
        \item The cokernel of $H_2(\Sigma;\Z)\to H_2(\overline{M};\Z)\cong H_2(M;\Z)$ is free of rank $b^2(M)-(l-1)$. In particular, $b^2(M)\geq l-1$.
        \item If $b^1(M)=0$, then $b^2(M)=l-1$, and every $\Sigma_j$ is diffeomorphic to $S^2$. In particular, if $b^1(M)=b^2(M)=0$, then $M$ has only one end, and $\Sigma\cong S^2$.
    \end{enumerate}
\end{lemma}
\begin{proof}
    Consider the exact sequence of the pair $(\overline{M},\Sigma)$ with integer coefficients,
    \begin{equation*}
        H_3(\overline{M},\Sigma)\to H_2(\Sigma)\to H_2(\overline{M})\to H_2(\overline{M},\Sigma)\to H_1(\Sigma)\to H_1(\overline{M}).
    \end{equation*}
    By Lefschetz duality, $H_3(\overline{M},\Sigma;\Z)\cong H^0(\overline{M};\Z)\cong\Z$, since $M$ is connected, and its image in $H_2(\Sigma;\Z)\cong\Z^l$ is spanned by $[\Sigma]=\sum_j[\Sigma_j]$. Hence the image of $H_2(\Sigma;\Z)$ in $H_2(\overline{M};\Z)$ has rank $l-1$. Again by Lefschetz duality, $H_2(\overline{M},\Sigma;\Z)\cong H^1(\overline{M};\Z)$ is free of rank $b^1(M)$. The cokernel of $H_2(\Sigma;\Z)\to H_2(\overline{M};\Z)$ is isomorphic to a subgroup of it, so it is free of rank $b^2(M)-(l-1)$. This proves (1).

    If $b^1(M)=0$, then $H_2(\overline{M},\Sigma;\Z)=0$, so this cokernel vanishes and $b^2(M)=l-1$. Moreover, $H_1(\Sigma;\Z)$ injects into $H_1(\overline{M};\Z)$, which is finite. Since $H_1(\Sigma;\Z)$ is free, it vanishes, i.e.\ every $\Sigma_j$ is a $2$-sphere. If in addition $b^2(M)=0$, then $l=1$.
\end{proof}

\section{Dirac Monopoles}\label{section: Dirac Monopoles}

\subsection{Local Models of $\U(1)$-Dirac Monopoles}\label{subsection: Local Models of U(1)-Dirac Monopoles}

In this subsection, we construct the Higgs field of the $\U(1)$-Dirac monopole, i.e.\ a harmonic function with prescribed behaviors near the singularities and along the ends.

\begin{definition}\label{definition:p1,...,pn,...,pk}
    Let $p_1, \cdots, p_s \in M$ be distinct points, and let $(k_1^I,\cdots,k_s^I)\in\Z^s$. For all $\psi \in C^\infty_c(M)$, define the \emph{current} by
    \begin{align*}
        \delta_{\mathrm{current}} : C^\infty_c(M) &\to \R \\
        \psi &\mapsto -\sum\limits_{i=1}^{s} k_i^I \psi (p_i ).
    \end{align*}
    We call $k_i^I$ the \emph{charge} of the singularity $p_i$.
\end{definition}

We fix $\varepsilon_0\in(0,1)$ such that the balls $B_{2\varepsilon_0}(p_i)$ are pairwise disjoint and contained in the domains of geodesic normal coordinates centered at $p_i$, and $R_0>2$ such that $K\cup\bigsqcup_{i} B_{2\varepsilon_0}(p_i)\subset\{\rho<R_0\}$, so that $\{\rho\geq R_0\}=\Upsilon([R_0,\infty)\times\Sigma)$.

\begin{theorem}\label{proposition:existence of harmonic function PhiD}
    Let $(M,g)$ be an asymptotically cylindrical $3$-manifold. Given $\overline{m}\in\R$, $\vec{k}^I=(k_1^I, \cdots, k_s^I) \in \Z^s$ and $\vec{k}^O=(k_1^O, \cdots, k_l^O) \in \Z^l$ satisfying $\sum\limits_{i=1}^{s} k_i^I = \sum\limits_{j=1}^l k_j^O $, there exists a unique harmonic function $\phi_D : M \setminus \{p_1,\cdots,p_s\} \to \R$ such that $\Delta\phi_D=4\pi\delta_{\mathrm{current}}$ in the sense of distributions, i.e.\ $\int_M\phi_D\Delta\psi\vol_g=4\pi\delta_{\mathrm{current}}(\psi)$ for all $\psi\in C_c^\infty(M)$, and such that $\phi_D$ has the following asymptotic expansions near the singularities and along the ends:
    \begin{align*}
         \phi_D &= m_i^I - \frac{k_i^I}{r_i} + O (r_i ), \qquad \text{as }r_i \to 0, \\
         \phi_D &= m_j^O + \frac{4\pi k_j^O}{ \vol(\Sigma_j)} \rho + O (e^{-\delta'\rho} ), \qquad \text{on }U(\Sigma_j),\text{ as }\rho \to \infty,
    \end{align*}
    where $r_i = \operatorname{dist} (\cdot, p_i )$, the numbers $m_i^I,m_j^O\in\R$ and $\delta'>0$ are constants, and the mass $\vec{m}^O=(m_1^O,\cdots,m_l^O)$ satisfies $\frac{1}{l}(m_1^O+\cdots+m_l^O)=\overline{m}$. The differential of the $O(r_i)$-term is bounded, and the derivatives of the $O(e^{-\delta'\rho})$-term are also $O(e^{-\delta'\rho})$. Moreover, $m_j^O-\overline{m}$ and $m_i^I-\overline{m}$ do not depend on $\overline{m}$. In particular, if $M$ has only one end, then $m_1^O=\overline{m}$.
\end{theorem}

\begin{proof}
    We fix $\beta\in(-\min\{\delta,\sqrt{\mu_1}\},0)$.
    \begin{enumerate}[label=\textbf{(\arabic*)},listparindent=\parindent]
        \item \textbf{Local models.} Let $x$ be geodesic normal coordinates centered at $p_i$, so that $r_i=\lvert x\rvert$, and put $\hat x:=\frac{x}{\lvert x\rvert}$. Using the Ricci tensor $\Ric_{p_i}$ and the scalar curvature $R_{p_i}$ at $p_i$, we set
        \begin{equation*}
            \widetilde{u}_i := \frac{1}{r_i} + c_i,\qquad c_i:=\frac{1}{12}\Big(\frac{\Ric_{p_i}(x,x)}{r_i}-R_{p_i}r_i\Big).
        \end{equation*}
        Then, as a distribution on $B_{2\varepsilon_0}(p_i)$, $\Delta\widetilde u_i=4\pi\delta_{p_i}+O(1)$, where the $O(1)$-term is bounded and smooth away from $p_i$. Indeed, the volume density in normal coordinates is $\theta=1-\frac16\Ric_{p_i}(x,x)+O(r_i^3)$, and by the Gauss lemma, $\Delta f=-f''-(\frac{2}{r_i}+\partial_{r_i}\log\theta)f'$ for radial functions $f$, so
        \begin{equation*}
            \Delta\frac{1}{r_i}=4\pi\delta_{p_i}+\frac{\partial_{r_i}\log\theta}{r_i^2}=4\pi\delta_{p_i}-\frac{\Ric_{p_i}(\hat x,\hat x)}{3r_i}+O(1).
        \end{equation*}
        On the other hand, $c_i=O(r_i)$ has second derivatives of order $O(r_i^{-1})$, so $\Delta c_i=\Delta_0c_i+O(r_i)$, where $\Delta_0$ is the Euclidean Laplacian, and a direct computation gives $\Delta_0c_i=\frac{\Ric_{p_i}(\hat x,\hat x)}{3r_i}$.

        Define the extension to $M$ by $u_i^I := - k_i^I  \chi_i^I \widetilde{u}_i$, where $\chi_i^I$ is a smooth cut-off function satisfying $\chi_i^I \equiv 1$ on $B_{\varepsilon_0}(p_i)$ and $\supp \chi_i^I \subset B_{2\varepsilon_0}(p_i)$. Consequently,
        \begin{equation}\label{eq:distr-uI}
            \Delta u_i^I = -4\pi k_i^I \delta_{p_i} + f_i^I,
        \end{equation}
        where $f_i^I$ is bounded, supported in $\overline{B_{2\varepsilon_0}(p_i)}$, and smooth on $M\setminus\{p_i\}$. For each end $U(\Sigma_j)$, we introduce
        \begin{equation*}
            u_j^O := \chi_j^O ( \overline{m} + \frac{4\pi k_j^O}{\vol(\Sigma_j)} \rho ),
        \end{equation*}
        where the smooth cut-off function $\chi_j^O$ satisfies $\chi_j^O=1$ on $U(\Sigma_j)\cap\{\rho\geq R_0+1\}$, $0\leq\chi_j^O\leq1$, and $\chi_j^O=0$ outside $U(\Sigma_j)\cap\{\rho>R_0\}$. Since the metric is asymptotically cylindrical, $\Delta u_j^O = O(e^{-\delta\rho})$. Now we set $u := \sum_{i=1}^{s} u_i^I + \sum_{j=1}^l u_j^O$. By \eqref{eq:distr-uI}, the distributional Laplacian of $u$ is
        \begin{equation*}
            \Delta u = 4\pi \delta_{\mathrm{current}} + f,
            \qquad
            f := \sum_{i=1}^{s} f_i^I + \sum_{j=1}^l \Delta u_j^O ,
        \end{equation*}
        and $f\in L^2_{0,\beta}(M)$, since $-\beta<\delta$.

        \item \textbf{Total flux.} We claim that $\int_M f \vol_g = 0$. Let $M_{\varepsilon,R} := \{\rho\leq R\}\setminus\bigsqcup_{i} B_{\varepsilon}(p_i)$, and let $\nu$ be its outward unit normal. Since $\Delta u=f$ on $M_{\varepsilon,R}$ and $\Delta = -\operatorname{div}\nabla$, Stokes' theorem gives
        \begin{equation*}
            \int_{M_{\varepsilon,R}} f\vol_g
            = -\sum_{i=1}^{s} \int_{\partial B_{\varepsilon}(p_i)} \langle du_i^I, \nu\rangle
               -\sum_{j=1}^l \int_{\{\rho=R\}\cap U(\Sigma_j)} \langle du_j^O, \nu\rangle.
        \end{equation*}
        On $\partial B_{\varepsilon}(p_i)$, we have $\nu = -\frac{\partial}{\partial r_i}$ and $u_i^I = - \frac{k_i^I}{r_i} + O(r_i)$, so $\langle du_i^I,\nu\rangle = - \frac{k_i^I}{r_i^2} + O(1)$. On the ends, we have $\nu = \frac{\partial}{\partial \rho}$ and $\langle du_j^O,\nu\rangle = \frac{4\pi k_j^O}{\vol(\Sigma_j)} + O(e^{-\delta\rho})$. Letting $\varepsilon\to0$ and $R\to\infty$, we obtain
        \begin{equation*}
            \int_M f \vol_g = 4\pi \sum_{i=1}^{s} k_i^I - 4\pi \sum_{j=1}^l k_j^O = 0,
        \end{equation*}
        where the last equality follows from the hypothesis.

        \item \textbf{Correction of the constant terms.} Let $\mathcal{K}$ be as in Lemma~\ref{lemma: d*d(L(M)) = {intM u vol=0}}, and for $h\in\mathcal{K}$ let $b_j(h)$ be the coefficient of $\rho$ in \eqref{equation: expansion of harmonic functions of linear growth}. Since $\Delta h=0$, Stokes' theorem on $\{\rho\leq R\}$ with $R>R_0+1$ gives $\int_M (\Delta\chi_j^O) h \vol_g=\int_{\{\rho=R\}\cap U(\Sigma_j)}\langle dh,\nu\rangle$, and letting $R\to\infty$, we obtain $\int_M (\Delta\chi_j^O) h \vol_g=b_j(h)\vol(\Sigma_j)$. By Lemma~\ref{lemma: d*d(L(M)) = {intM u vol=0}}(2), the linear map $\mathcal{K}\to\R^l$, $h\mapsto(b_1(h)\vol(\Sigma_1),\cdots,b_l(h)\vol(\Sigma_l))$, vanishes exactly on the constant functions, and its image is the hyperplane $\{y\in\R^l\mid y_1+\cdots+y_l=0\}$, since $\dim\mathcal{K}=l$. Since $\int_M f\vol_g=0$, the linear functional $h\mapsto\int_M fh\vol_g$ on $\mathcal{K}$ also vanishes on the constant functions. Hence there is a unique $(\tau_1,\cdots,\tau_l)\in\R^l$ with $\tau_1+\cdots+\tau_l=0$ such that
        \begin{equation*}
            \int_M f h\vol_g=-\sum_{j=1}^l \tau_j b_j(h)\vol(\Sigma_j)\qquad\text{for all } h\in\mathcal{K}.
        \end{equation*}
        Then $f' := f+\sum_{j=1}^l \tau_j\Delta\chi_j^O$ belongs to $L^2_{0,\beta}(M)$ and satisfies $\int_M f'h\vol_g=0$ for all $h\in\mathcal{K}$. Hence Lemma~\ref{lemma: d*d(L(M)) = {intM u vol=0}}(1) provides a function $v \in L^2_{2,\beta}(M)$ such that $\Delta v = -f'$. By elliptic regularity, $v$ is of class $C^1$ near $p_i$, since $f'$ is bounded there. Along the ends, $\Delta v$ and all its derivatives are $O(e^{-\delta\rho})$, and interior elliptic estimates on balls of unit size show that $v$ and all its derivatives are $O(e^{\frac12\beta\rho})$.

        \item \textbf{Existence.} Define $\phi_D := u + \sum_{j=1}^l \tau_j\chi_j^O + v$. Then, in the sense of distributions,
        \begin{equation*}
            \Delta \phi_D = (4\pi \delta_{\mathrm{current}} + f) + (f'-f) - f' = 4\pi \delta_{\mathrm{current}},
        \end{equation*}
        so $\phi_D$ is harmonic on $M\setminus\{p_1,\cdots,p_s\}$. Near $p_i$, we have $\phi_D=u_i^I+v=-k_i^I(r_i^{-1}+c_i)+v$, hence $\phi_D=m_i^I-\frac{k_i^I}{r_i}+O(r_i)$ with $m_i^I:=v(p_i)$, and the differential of the $O(r_i)$-term is bounded, since $v$ is of class $C^1$ and $\lvert dc_i\rvert\leq C$. On $U(\Sigma_j)\cap\{\rho\geq R_0+1\}$, we have
        \begin{equation*}
            \phi_D = m_j^O + \frac{4\pi k_j^O}{\vol(\Sigma_j)} \rho + v,\qquad m_j^O:=\overline{m}+\tau_j,
        \end{equation*}
        where $v$ and all its derivatives are $O(e^{-\delta'\rho})$ for $\delta':=-\frac12\beta$, and $\frac{1}{l}(m_1^O+\cdots+m_l^O)=\overline{m}$, since $\tau_1+\cdots+\tau_l=0$.

        \item \textbf{Uniqueness.} Let $\phi_D'$ be another harmonic function with the above expansions. Then $h:=\phi_D-\phi_D'$ is harmonic on $M\setminus\{p_1,\cdots,p_s\}$ and bounded near each $p_i$, so it extends to a harmonic function on $M$. Moreover, $h$ is bounded on the ends, so $h\in\mathcal{K}$ and $b_j(h)=0$ for all $j$. By Lemma~\ref{lemma: d*d(L(M)) = {intM u vol=0}}(2), $h$ is constant, and $h=0$, since the average masses of $\phi_D$ and $\phi_D'$ are both $\overline{m}$. In particular, $\phi_D+c$ is the harmonic function for the average mass $\overline{m}+c$, so $m_j^O-\overline{m}$ and $m_i^I-\overline{m}$ do not depend on $\overline{m}$.\qedhere
    \end{enumerate}
\end{proof}

\begin{remark}\label{remark: masses are determined}
    If $M$ has more than one end, the masses $m_1^O,\cdots,m_l^O$ cannot be prescribed independently. For example, let $M=\R_t\times S^2$ with the product metric $dt^2+g_{S^2}$, where $g_{S^2}$ is the round metric of area $4\pi$, let $\rho=\lvert t\rvert$ on the ends $\{t>1\}$ and $\{t<-1\}$, and let $s=1$, $p_1=(t_0,x_0)$ and $k_1^I=1$. Since $\Delta=-\partial_t^2+\Delta_{S^2}$, the average of $\phi_D$ over $\{t\}\times S^2$ equals $\frac{1}{2}\lvert t-t_0\rvert+c+c't$ for some constants $c,c'\in\R$, and the part of $\phi_D$ orthogonal to the constants on each $\{t\}\times S^2$ decays exponentially on the ends. Hence the masses $m_1^O$ on $\{t>1\}$ and $m_2^O$ on $\{t<-1\}$ satisfy $m_2^O-m_1^O=t_0$, whatever the charges $k_1^O,k_2^O$ of the ends (with $k_1^O+k_2^O=1$) are. In particular, the masses do not determine the charges of the ends. Moreover, the masses depend on the choice of the distance function: if $\rho$ is replaced by $\rho+\rho_0$ on the $j$-th end for a constant $\rho_0$, then the mass $m_j^O$ of the same $\phi_D$ is replaced by $m_j^O-\frac{4\pi k_j^O}{\vol(\Sigma_j)}\rho_0$, while the charge $k_j^O$, i.e.\ the flux of $\frac{1}{4\pi}*d\phi_D$ through the $j$-th end, does not change. For these reasons, we prescribe the charges $\vec k^O$ of the ends and the average mass $\overline{m}$ in Theorem~\ref{proposition:existence of harmonic function PhiD}, and not the masses as in \cite{Oliveira_2016}.
\end{remark}

\subsection{Mass of $\U(1)$-Dirac Monopoles}\label{subsection: Mass of $U(1)$-Dirac Monopoles}

From now on, we assume $s=n$ and $k_1^I=\cdots=k_{n}^I=1$. Recall that $H^0 ( \Sigma; \R ) \cong \R^l$ with $l= b^0(\Sigma)$.

Define the \emph{mass} $\vec{m}^O$ of the $\U(1)$-Dirac monopole $(a_D,\phi_D)$ by
\begin{equation*}
    \vec{m}^O=(m_1^O,\cdots,m_l^O)\in \R^l.
\end{equation*}
On the ACyl $3$-manifold $(M,g)$, one method to construct a new $\U(1)$-Dirac monopole $(a^\lambda_D,\phi_D^\lambda)$ with mass
\begin{equation*}
    \vec{m}^O +(\lambda-1):=(m_1^O+\lambda-1,\cdots,m_l^O+\lambda-1)
\end{equation*}
is to add the constant $\lambda-1$ to the Higgs field:
\begin{equation*}
    (a_D^{\lambda}(x),\phi_D^{\lambda}(x)) := (a_D(x), \phi_D(x)+ (\lambda-1)),
\end{equation*}
where $\lambda$ is a scaling factor. We call $(a_D^\lambda,\phi_D^\lambda)$ a scaled $\U(1)$-Dirac monopole, since on $\R^3$, adding $\lambda-1$ to the Higgs field of the Dirac monopole of mass $1$ is the same as scaling it by $\lambda$.

In this paper, this scaling factor $\lambda=\lambda(\overline{m})$ is related to the following definition of $\overline{m}$. From now on, we fix charges $\vec k^O=(k_1^O,\cdots,k_l^O)\in\Z^l$ of the ends with $k_1^O+\cdots+k_l^O=n$, and the $\U(1)$-Dirac monopole $(a_D,\phi_D)$ of Theorem~\ref{proposition:existence of harmonic function PhiD} with these charges and average mass $1$, i.e.\ $\frac{1}{l}(m_1^O+\cdots+m_l^O)=1$. If $M$ has only one end, then $k_1^O=n$ and $m_1^O=1$. Define the \emph{average mass} $\overline{m}$ of the scaled $\U(1)$-Dirac monopole $(a_D^\lambda,\phi_D^\lambda)$ by
\begin{equation*}
    \overline{m} = \frac{1}{l} [ (m_1^O+\lambda-1) + \cdots + (m_{l}^O+\lambda-1) ] = \lambda.
\end{equation*}

Note that $\phi^{\overline{m}}_D :=\phi_D + \overline{m}-1$ is a harmonic function on $M \setminus \{p_1,\cdots,p_n\}$ with mass $\vec{m}^O+ \overline{m}-1$ and average mass $\overline{m}$,
where $\vec{m}^O$ does not depend on $\overline{m}$. In the rest of this paper, we regard $\overline{m}$ as a parameter, which will be chosen sufficiently large.

The following lemma implies that the $\U(1)$-Dirac monopole has a sufficiently large Higgs field outside every small singularity neighborhood.

\begin{lemma}\label{lemma:large Higgs field}
    Let $(M,g)$ be an asymptotically cylindrical $3$-manifold with $b^2(M)=0$. If the average mass $\overline{m}$ is sufficiently large, then
    \begin{align*}
        \phi_D^{\overline{m}} \geq \frac{\overline{m}}{2}
    \end{align*}
    on $U(\varepsilon)= M \setminus \overline{\bigsqcup\limits_{i=1}^{n} B_{\varepsilon}(p_i)}$, where $\varepsilon = \sqrt{\frac{5}{2 \overline{m}}}=O(\overline{m}^{-\frac{1}{2}})$.
\end{lemma}

\begin{proof}
    Since $k_1^I=\cdots=k_n^I=1$, the Higgs field tends to negative infinity at $\{p_1,\cdots,p_n\}$, and to positive infinity along the unique cylindrical end (Lemma~\ref{lemma: topology of the ends}).
    \begin{enumerate}[label=\textbf{(\arabic*)},listparindent=\parindent]

            \item \textbf{Near the end}

    Since $k_1^O=n> 0$ in our assumption and the number of singularities $\{p_1,\cdots,p_{n}\}$ is finite, we can choose $R=\frac{\vol (\Sigma)}{4\pi k_1^O}\overline{m}$ sufficiently large by increasing $\overline{m}$, so that the unique end $\Upsilon([R,\infty)_r\times \Sigma)$ contains no singularity and
    \begin{align*}
         \phi_D^{\overline{m}} &= m_1^O +\frac{4\pi k_1^O}{\vol (\Sigma)}\rho+\overline{m}-1 + O(e^{-\delta'\rho}) \\
         &\geq m^O_1+\frac{4\pi k_1^O}{\vol (\Sigma)}R+\overline{m}-1-C \qquad (\text{constant }C>0)\\
         &=  m_1^O +\overline{m}+\overline{m}-1 -C \\
         &\geq \frac{\overline{m}}{2}, \qquad
         (\text{as } \overline{m} \to \infty)
    \end{align*}
    on $\Upsilon([R,\infty)_r\times \Sigma)$, where $m_1^O=1$ and $C$ do not depend on $\overline{m}$.

        \item \textbf{Near singularities}

         On each neighborhood $B_{\varepsilon_0}(p_i)$, we have $\phi_D(r_i)=-\frac{1}{r_i}+m_i^I+O(r_i)$.
        By adding the constant $\overline{m}-1$, we obtain
        \begin{equation*}
            \phi^{\overline{m}}_D(r_i) = -\frac{1}{r_i}+m_i^I+\overline{m}-1+O(r_i).
        \end{equation*}
Since $m_i^I$ is independent of $\overline{m}$ and $r_i<\varepsilon_0<1$, the term $(O(r_i)-1)$ is bounded below by a constant, which we absorb into $m_i^I$ to get a new constant $m^{II}_i$ as follows:
        \begin{equation*}
            \phi^{\overline{m}}_D(r_i) \geq -\frac{1}{r_i}+m_i^{II}+\overline{m},
        \end{equation*}
        where the right hand side is increasing along the radial coordinate.

        By increasing $\overline{m}$, we want to find $\varepsilon' < \varepsilon$ such that $\phi_D^{\overline{m}}(r_i)\geq \frac{\overline{m}}{2}$ for all $r_i \in (\varepsilon',\varepsilon)$. Since $m_i^I$'s do not change with $\overline{m}$, as we increase $\overline{m}$, we have inequalities
        \begin{equation*}
            -\frac{1}{r_i}+m_i^{II}+\overline{m} \geq -\frac{1}{r_i}+\frac{9}{10} \overline{m} \geq \frac{\overline{m}}{2},
        \end{equation*}
        whose solutions
        \begin{equation*}
            r_i \geq \frac{5}{2 \overline{m}} ,\qquad
            \overline{m} \geq 10\max\{ -m_1^{II},\cdots,-m_n^{II}\}
        \end{equation*}
        satisfy $\phi_D^{\overline{m}}(r_i) \geq \frac{\overline{m}}{2}$. For sufficiently large $\overline{m}$, we have $\varepsilon':=\frac{5}{2 \overline{m}}< \varepsilon=\sqrt{\frac{5}{2 \overline{m}}}<\varepsilon_0$. Hence $\phi_D^{\overline{m}}(r_i)\geq\frac{\overline{m}}{2}$ at every point with $\varepsilon'\leq r_i\leq\varepsilon$. In particular,
        \begin{equation*}
            \phi_D^{\overline{m}} |_{\partial  B_\varepsilon(p_i)} \geq \frac{\overline{m}}{2}.
        \end{equation*}

    \item\textbf{On the compact neck}

    Since $\phi_D^{\overline{m}}$ is a harmonic function on the compact neck $M \setminus \{\bigsqcup\limits_{i=1}^n B_{\varepsilon}(p_i) \bigcup \Upsilon([R,\infty) \times \Sigma)\}$, the maximum principle implies that the minimum can only be attained at the boundaries:
    \begin{equation*}
        \phi_D^{\overline{m}} |_{\bigsqcup\limits_{i=1}^n \partial B_{\varepsilon}(p_i)} \geq \frac{\overline{m}}{2}, \qquad   \phi_D^{\overline{m}} |_{ \Upsilon(\{R\} \times \Sigma)} \geq \frac{\overline{m}}{2},
    \end{equation*}
    so that
    \begin{equation*}
         \phi_D^{\overline{m}} |_{M \setminus \{\bigsqcup\limits_{i=1}^n B_{\varepsilon}(p_i) \bigcup \Upsilon([R,\infty) \times \Sigma)\}} \geq \frac{\overline{m}}{2}.
    \end{equation*}
\end{enumerate}
This completes the proof.
\end{proof}

\begin{remark}
    The assumption $b^2(M)=0$ in Lemma~\ref{lemma:large Higgs field} is only used to ensure that $M$ has only one end, on which $k_1^O=n>0$. If $M$ has several ends, the conclusion still holds if $k_j^O\geq0$ for all $j$, whereas the lemma fails if $k_j^O<0$ for some $j$, since then $\phi_D^{\overline{m}}\to-\infty$ along the $j$-th end.
\end{remark}

\subsection{Construction of $\SU(2)$-Dirac Monopoles}\label{subsection: Construction of SU(2)-Dirac Monopoles}

In this subsection, we construct $\U(1)$-Dirac monopoles with singularities $S_p=\{p_1,\cdots,p_s\}$ on the asymptotically cylindrical $3$-manifold, and lift them to $\SU(2)$-Dirac monopoles.

In this paper, we identify the Lie algebra $\mathfrak{u}(1)=i\R$ of $\U(1)$ with $\R$ by $\frac{i}{2}t\leftrightarrow t$, i.e.\ $\U(1)\cong\R/4\pi\Z$ by $e^{\frac{i}{2}t}\leftrightarrow t$. Then the curvature $F_a$ of a connection $a$ on a principal $\U(1)$-bundle is a closed real $2$-form, and a closed real $2$-form on a manifold $X$ is the curvature of a connection on a principal $\U(1)$-bundle over $X$ if and only if its periods lie in $4\pi\Z$, i.e.\ its class vanishes in $H^2(X;\R/4\pi\Z)$. We use this identification, since the differential of the inclusion $\U(1)\hookrightarrow\SU(2)$, $e^{i\theta}\mapsto\operatorname{diag}(e^{i\theta},e^{-i\theta})$, maps $1\in\R$ to $\frac{\sigma_3}{2}$, which has norm $1$ (Subsection~\ref{subsection: PS Monopoles}). Hence a $\U(1)$-Dirac monopole and the induced $\SU(2)$-Dirac monopole have Higgs fields of the same norm (Theorem~\ref{theorem:construction of Dirac monopoles on SU(2) bundle}).

\begin{definition}[$\U(1)$-Dirac Monopoles]\label{definition: U(1)-Dirac monopoles}
    Let $a$ be a connection on a principal $\U(1)$-bundle $L'\to M\setminus S_p$, and let $\phi$ be a real function on $M\setminus S_p$. The pair $(a,\phi)$ is called a \emph{(singular) $\U(1)$-Dirac monopole} with singularities $p_1,\cdots,p_s$ and charges $k_1^I,\cdots,k_s^I$ if it satisfies the Bogomolny equation $*F_a=d\phi$, and $\phi=-\frac{k_i^I}{r_i}+O(1)$ as $r_i\to0$.
\end{definition}

\begin{theorem}\label{theorem:well definedness of b(p,[mO])}
    Let $p_1,\cdots,p_s\in M$ be distinct points, let $\vec k^I\in\Z^s$ and $\vec k^O\in\Z^l$ satisfy $\sum_ik_i^I=\sum_jk_j^O$, and let $\phi_D$ be the harmonic function of Theorem~\ref{proposition:existence of harmonic function PhiD}.
    \begin{enumerate}[label=\textbf{(\arabic*)},listparindent=\parindent]
        \item There is a unique class $b=b(p_1,\cdots,p_s;\vec k^I,\vec k^O)\in H^2(M;\R/4\pi\Z)$ whose restriction to $M\setminus S_p$ is the class $[*d\phi_D]\in H^2(M\setminus S_p;\R/4\pi\Z)$. It does not depend on the average mass $\overline m$.
        \item There is a connection $a_D$ on a principal $\U(1)$-bundle $L'\to M\setminus S_p$ such that $(a_D,\phi_D)$ is a $\U(1)$-Dirac monopole with singularities $p_1,\cdots,p_s$ and charges $k_1^I,\cdots,k_s^I$ if and only if $b=0$. The pair $(L',a_D)$ is unique up to gauge equivalence and tensoring with flat $\U(1)$-bundles, i.e.\ up to the action of $H^1(M\setminus S_p;\R/4\pi\Z)\cong H^1(M;\R/4\pi\Z)$.
        \item If $b^2(M)=l-1$, in particular if $b^1(M)=0$ or $b^2(M)=0$, then $b=0$ for all choices of the singularities and the charges.
    \end{enumerate}
\end{theorem}

\begin{proof}
    \begin{enumerate}[label=\textbf{(\arabic*)},listparindent=\parindent]
        \item Let $G=\R/4\pi\Z$, and set $U_i = B_\varepsilon(p_i) \setminus \{p_i\}$ for small $\varepsilon>0$. By excision, $H^k(M,M\setminus S_p;G)\cong\bigoplus_{i=1}^sH^{k-1}(U_i;G)$ for $k\geq2$, which vanishes for $k=2$. Moreover, $H^3(M;G)=0$, since $M$ is homotopy equivalent to $\overline M$, which has nonempty boundary. Hence the long exact sequence of the pair $(M,M\setminus S_p)$ reduces to
        \begin{equation}
            0 \to H^2(M; G) \to H^2(M \setminus S_p; G) \to \bigoplus_{i=1}^s H^2(U_i; G) \to 0,
            \label{equation:0 to H2(M) to H2(U) to oplus H2(Ui) to 0}
        \end{equation}
        where the first map is the restriction, and the second one is given by the periods over the spheres $\partial B_\varepsilon(p_i)$. Since $\phi_D$ is harmonic, $*d\phi_D$ is closed on $M\setminus S_p$, and by the expansion of $\phi_D$ near $p_i$ and Stokes' theorem, its period over $\partial B_\varepsilon(p_i)$ is $4\pi k_i^I\in4\pi\Z$. Hence the class of $*d\phi_D$ in $H^2(M\setminus S_p;G)$ is the restriction of a unique class $b$. Since adding a constant to $\phi_D$ does not change $*d\phi_D$, $b$ does not depend on $\overline m$.

        \item By the identification above, a connection $a_D$ with curvature $*d\phi_D$ exists if and only if the class of $*d\phi_D$ in $H^2(M\setminus S_p;\R/4\pi\Z)$ vanishes, i.e.\ $b=0$ by (1). Two such pairs $(L',a_D)$ differ by tensoring with a flat $\U(1)$-bundle, and flat $\U(1)$-bundles over $M\setminus S_p$ are classified by their holonomies, i.e.\ by $H^1(M\setminus S_p;\R/4\pi\Z)\cong H^1(M;\R/4\pi\Z)$, since $S_p$ is a finite set of points in a $3$-manifold.

        \item By Lemma~\ref{lemma: topology of the ends}, $b^2(M)=l-1$ holds if $b^1(M)=0$ or $b^2(M)=0$. Suppose that $b^2(M)=l-1$. Then the map $H_2(\Sigma;\Z)\to H_2(M;\Z)$ is surjective by Lemma~\ref{lemma: topology of the ends}(1). By the Mayer-Vietoris sequence for $M=(M\setminus S_p)\bigcup(\bigsqcup_{i} B_\varepsilon(p_i))$, the map $H_2(M\setminus S_p;\Z)\to H_2(M;\Z)$ is surjective, and its kernel is generated by the spheres $\partial B_\varepsilon(p_i)$. Therefore $H_2(M\setminus S_p;\Z)$ is generated by the spheres $\partial B_\varepsilon(p_i)$ and the surfaces $\{\rho=R\}\cap U(\Sigma_j)$, $j=1,\cdots,l$, for sufficiently large $R$. The periods of $*d\phi_D$ over $\partial B_\varepsilon(p_i)$ are $\pm4\pi k_i^I$ as above, and its periods over $\{\rho=R\}\cap U(\Sigma_j)$ are $\pm4\pi k_j^O$ by the expansion of $\phi_D$ on the $j$-th end. All of them lie in $4\pi\Z$. Therefore $[*d\phi_D]=0$ in $H^2(M\setminus S_p;\R/4\pi\Z)$, i.e.\ $b=0$.\qedhere
    \end{enumerate}
\end{proof}

\begin{proposition}[admissible singularities]\label{proposition: admissible singularities}
    Let $(M,g)$ be an asymptotically cylindrical $3$-manifold with $l$ ends, and put $N:=b^2(M)-(l-1)\geq0$. Let $\vec k^I=(k_1^I,\cdots,k_s^I)\in\Z^s$ and $\vec k^O=(k_1^O,\cdots,k_l^O)\in\Z^l$ satisfy $\sum_ik_i^I=\sum_jk_j^O$, and suppose that $k_1^I,\cdots,k_N^I$ are nonzero. Then for all distinct points $p_{N+1},\cdots,p_s\in M$, there are points $p_1,\cdots,p_N\in M$ such that $p_1,\cdots,p_s$ are distinct and $b(p_1,\cdots,p_s;\vec k^I,\vec k^O)=0$.
\end{proposition}
\begin{proof}
    If $N=0$, then $b=0$ by Theorem~\ref{theorem:well definedness of b(p,[mO])}(3). Let $N\geq1$.
    \begin{enumerate}[label=\textbf{(\arabic*)},listparindent=\parindent]
        \item \textbf{Reduction to periods.} By Lemma~\ref{lemma: topology of the ends}(1), there are closed oriented embedded surfaces $S_1,\cdots,S_N\subset M$ whose classes map to a basis of the cokernel of $H_2(\Sigma;\Z)\to H_2(M;\Z)$. As in the proof of Theorem~\ref{theorem:well definedness of b(p,[mO])}(3), $H_2(M\setminus S_p;\Z)$ is generated by the spheres $\partial B_\varepsilon(p_i)$, the cross-sections of the ends, and closed oriented surfaces $S_a'\subset M\setminus S_p$ homologous to $S_a$ in $M$. The periods of $*d\phi_D$ over the first two kinds of surfaces lie in $4\pi\Z$. Hence the class of $P_a:=\int_{S_a'}*d\phi_D$ in $\R/4\pi\Z$ does not depend on the choice of $S_a'$, and $b=0$ if and only if $P_a\in4\pi\Z$ for $a=1,\cdots,N$.

        \item \textbf{The period maps.} For $p\in M$, let $\psi_p$ be the harmonic function of Theorem~\ref{proposition:existence of harmonic function PhiD} with the single singularity $p$ of charge $1$, charges $(1,0,\cdots,0)$ of the ends and average mass $0$, and let $\psi_0$ be the one without singularities, with charges $\vec k^O-(\sum_ik_i^I)(1,0,\cdots,0)$ of the ends and average mass $0$. By the uniqueness in Theorem~\ref{proposition:existence of harmonic function PhiD}, $\phi_D=\sum_{i=1}^sk_i^I\psi_{p_i}+\psi_0+\overline m$, so $P_a\equiv\sum_{i=1}^sk_i^IP_a(p_i)+P_a^0$ modulo $4\pi$, where
        \begin{equation*}
            P_a(p):=\int_{S_a'}*d\psi_p\ \ \mathrm{mod}\ 4\pi,\qquad P_a^0:=\int_{S_a}*d\psi_0,
        \end{equation*}
        and $S_a'\subset M\setminus\{p\}$ is homologous to $S_a$ in $M$. As in (1), $P_a(p)\in\R/4\pi\Z$ is well defined.

        The map $P_a:M\to\R/4\pi\Z$ is continuous. Indeed, for $p$ in a compact subset of $M$, the proof of Theorem~\ref{proposition:existence of harmonic function PhiD} for $\psi_p$ can be carried out with the same $\varepsilon_0$, $R_0$ and $\chi_j^O$, and with the cut-off function $\chi(\varepsilon_0^{-1}\operatorname{dist}(\cdot,p))$ near $p$ for a fixed function $\chi$. Then $f$ is bounded uniformly in $p$, and it depends continuously on $p$ in $L^2_{0,\beta}(M)$ and in $L^q$ on compact subsets of $M$ for $q<\infty$. Hence so do the constants $\tau_j$ and $f'$, and by Lemma~\ref{lemma: d*d(L(M)) = {intM u vol=0}}(1) and interior $L^q$-estimates, $v$ depends continuously on $p$ in $C^1$ on compact subsets of $M$. Thus $\int_S*d\psi_p$ depends continuously on $p\in M\setminus S$ for every closed surface $S\subset M$, and we may use the same surface $S_a'$ for all $p$ near a given point.

        Let $\gamma\subset M$ be a closed curve transverse to $S_a$. On $\gamma\setminus S_a$, $P_a$ lifts to the continuous function $p\mapsto\int_{S_a}*d\psi_p$. Let $q\in\gamma\cap S_a$, and let $S''$ be obtained by pushing $S_a$ slightly to one side near $q$. Then $\int_{S''}*d\psi_p$ is continuous in $p$ near $q$, and it differs from $\int_{S_a}*d\psi_p$ by $0$ on one side of $S_a$ and by $\pm4\pi$, the flux of $*d\psi_p$ through a small sphere around $p$, on the other side. Hence the lift jumps by $\pm4\pi$ at each intersection point, and $P_a|_\gamma$ has degree $\pm(S_a\cdot\gamma)$, where $S_a\cdot\gamma$ is the intersection number and the sign is universal.

        \item \textbf{A degree argument.} The images of $[S_1],\cdots,[S_N]$ in $H_2(\overline M,\Sigma;\Z)\cong H^1(\overline M;\Z)\cong\operatorname{Hom}(H_1(\overline M;\Z),\Z)$ are linearly independent, since the map $H_2(\overline M;\Z)\to H_2(\overline M,\Sigma;\Z)$ has kernel equal to the image of $H_2(\Sigma;\Z)$. Hence there are closed curves $\gamma_1,\cdots,\gamma_N\subset M$ such that the matrix $(S_a\cdot\gamma_c)_{a,c=1}^N$ is invertible over $\mathbb Q$. By general position, we may assume that $\gamma_1,\cdots,\gamma_N$ are embedded, pairwise disjoint, transverse to $S_1,\cdots,S_N$, and disjoint from $p_{N+1},\cdots,p_s$. Consider the continuous map
        \begin{equation*}
            \mathcal P:\gamma_1\times\cdots\times\gamma_N\to(\R/4\pi\Z)^N,\qquad \mathcal P(p_1,\cdots,p_N):= \Big(\sum_{i=1}^sk_i^IP_a(p_i)+P_a^0 \Big)_{a=1}^N,
        \end{equation*}
        where $p_{N+1},\cdots,p_s$ are fixed. By (2), the map induced by $\mathcal P$ on the first homology is given by the matrix $(\pm k_c^I(S_a\cdot\gamma_c))_{a,c}$, whose determinant $\pm k_1^I\cdots k_N^I\det(S_a\cdot\gamma_c)$ is nonzero. Since the degree of a map between $N$-tori is the determinant of the induced map on the first homology, $\mathcal P$ is surjective. Thus there are $p_c\in\gamma_c$, $c=1,\cdots,N$, with $\mathcal P(p_1,\cdots,p_N)=0$, i.e.\ $b=0$ by (1). The points $p_1,\cdots,p_s$ are distinct, since the curves $\gamma_c$ are pairwise disjoint and avoid $p_{N+1},\cdots,p_s$.\qedhere
    \end{enumerate}
\end{proof}

\begin{remark}\label{remark: obstruction is sharp}
    The condition $b^2(M)=l-1$ in Theorem~\ref{theorem:well definedness of b(p,[mO])}(3) cannot be dropped in general. Suppose that $b^2(M)>l-1$ and $k_1^I\neq0$, and fix $p_2,\cdots,p_s$. Let $S_1$ be as in the above proof, and let $\gamma$ be a closed curve transverse to $S_1$ and disjoint from $p_2,\cdots,p_s$ with $S_1\cdot\gamma\neq0$. By the above proof, $p_1\mapsto P_1$ has nonzero degree $\pm k_1^I(S_1\cdot\gamma)$ on $\gamma$, so $P_1\notin4\pi\Z$, i.e.\ $b\neq0$, for some $p_1\in\gamma$, and there is no $\U(1)$-Dirac monopole with these data. This happens, for instance, on $M=(S^1\times S^2)\setminus\overline B$ with $S_1=\{\theta\}\times S^2$ and $\gamma=S^1\times\{x\}$. This is consistent with \cite[Remark~6]{Oliveira_2016}, where, on an asymptotically conical $3$-manifold $X$, the admissible positions of the singularities form the zero set of the corresponding map $b_{[m]}$, which is all of $X^k$ if $b^2(X)=0$.
\end{remark}

\begin{theorem}\label{theorem:construction of Dirac monopoles on SU(2) bundle}
    Suppose that $b(p_1,\cdots,p_s;\vec k^I,\vec k^O)=0$, which holds in particular if $b^2(M)=l-1$, and let $(L',a_D)$ be as in Theorem~\ref{theorem:well definedness of b(p,[mO])}(2). Let $P\to M\setminus S_p$ be the principal $\SU(2)$-bundle associated to $L'$ via the inclusion $\U(1)\hookrightarrow\SU(2)$, $e^{i\theta}\mapsto\operatorname{diag}(e^{i\theta},e^{-i\theta})$, and let $A_D$ be the connection on $P$ induced by $a_D$. Then the adjoint bundle splits as $\mathfrak g_P=\underline{\R}\oplus L$, where $\underline{\R}$ is spanned by the section $\hat\sigma$ corresponding to $\frac{\sigma_3}{2}$, and $L$ is the complex line bundle associated to $L'$ with weight $2$. The section $\hat\sigma$ is parallel with $\lvert\hat\sigma\rvert=1$, and $(A_D,\Phi_D)$ with $\Phi_D:=\phi_D\hat\sigma$ is an $\SU(2)$-Dirac monopole, i.e.\ it satisfies the Bogomolny equation~\eqref{equation: Bogomolny equation} on $M\setminus S_p$. Moreover, $\lvert\Phi_D\rvert=\lvert\phi_D\rvert$.
\end{theorem}
\begin{proof}
    The adjoint representation of $\SU(2)$, restricted to $\U(1)$, is trivial on $\R\frac{\sigma_3}{2}$ and has weight $2$ on its orthogonal complement, which gives the splitting of $\mathfrak g_P$. Since $A_D$ is induced from a $\U(1)$-connection, $\hat\sigma$ is parallel, and $\lvert\hat\sigma\rvert=\lvert\frac{\sigma_3}{2}\rvert=1$. Since the differential of the inclusion maps $1\in\R\cong\mathfrak{u}(1)$ to $\frac{\sigma_3}{2}$, we have $F_{A_D}=F_{a_D}\otimes\hat\sigma=*d\phi_D\otimes\hat\sigma$, and since $\hat\sigma$ is parallel, $d_{A_D}\Phi_D=d\phi_D\otimes\hat\sigma=*F_{A_D}$. Hence the pair $(A_D,\Phi_D)$ is the desired $\SU(2)$-Dirac monopole.
\end{proof}

\section{Approximate Solutions}\label{section: Approximate Solutions}

In this section, we construct the approximate solution and estimate its error term. Our method is to glue pull-back scaled PS monopoles onto the scaled $\SU(2)$-Dirac monopole
\begin{equation*}
    (A_D^{\overline{m}},\Phi_D^{\overline{m}}):=(A_D,\phi_D^{\overline{m}}\hat\sigma)
\end{equation*}
induced from $\phi_D^{\overline{m}}$ as in Theorem~\ref{theorem:construction of Dirac monopoles on SU(2) bundle}. Since $\lvert\Phi_D^{\overline{m}}\rvert=\lvert\phi_D^{\overline{m}}\rvert$, the average mass of $(A_D^{\overline{m}},\Phi_D^{\overline{m}})$ is also $\overline{m}$. Throughout this section, $\vec k^I=(1,\cdots,1)$ and $b(p_1,\cdots,p_n;\vec k^I,\vec k^O)=0$.

\subsection{PS Monopoles}\label{subsection: PS Monopoles}

Let
\begin{equation*}
        \sigma_1=
        \begin{pmatrix}
            0 & i \\
            i & 0 \\
        \end{pmatrix},
        \quad
        \sigma_2=
        \begin{pmatrix}
            0 & 1 \\
            -1 & 0 \\
        \end{pmatrix},
        \quad
        \sigma_3=
        \begin{pmatrix}
            i & 0 \\
            0 & -i \\
        \end{pmatrix}.
\end{equation*}
The tuple $\{\sigma_1, \sigma_2, \sigma_3\}$ is a basis of $\su(2)$ and satisfies
\begin{gather*}
    [ \frac{\sigma_i}{2}, \frac{\sigma_j}{2} ] = -\sum_{k=1}^3\varepsilon_{ijk} \frac{\sigma_k}{2}, \qquad
    [ \frac{\sigma_3}{2}, \frac{\sigma_1}{2} \pm i \frac{\sigma_2}{2} ] = \pm i(\frac{\sigma_1}{2} \pm i \frac{\sigma_2}{2}),
\end{gather*}
where $\varepsilon_{ijk}$ is the Levi-Civita symbol. We use the $\Ad$-invariant inner product $\langle X, Y \rangle = -2 \tr ( XY )$ on $\su(2)$, which is $-\frac12$ times the Killing form, and for which $\frac{\sigma_1}{2},\frac{\sigma_2}{2},\frac{\sigma_3}{2}$ is an orthonormal basis. Hence $\ad_X = [ X, \cdot ]$ is skew-adjoint for all $X \in \su(2)$. Since $[X,Y]$ corresponds to $-X\times Y$ under the isometry $\su(2)\cong\R^3$ given by this basis, we have
\begin{align}
    \lvert [ X, Y ] \rvert \leq \lvert X \rvert \lvert Y \rvert,\qquad\text{with equality if }\langle X,Y\rangle=0. \label{equation:|[X3s3,Y1s1+Y2s2]|=|X3s3||Y1s1+Y2s2|}
\end{align}

\begin{definition}[Mass and Charges]
    The \emph{mass} $m$ and the \emph{charge} $k$ of a monopole $(A,\Phi)$ on $\R^3$ are defined by
    \begin{equation*}
        m := \lim\limits_{\lvert x \rvert \to \infty} \lvert \Phi(x) \rvert \in \R_{\geq 0},
        \qquad
        k := \lim _{r \to \infty}\deg ( \frac{\Phi}{\lvert \Phi \rvert} : S_r(0)\to S^2) \in \Z ,
    \end{equation*}
    where $S^2$ is the unit sphere in $\su(2)\cong\R^3$.
    \label{definition: Mass and Charges of Monopoles}
\end{definition}

On an asymptotically cylindrical $3$-manifold, the Higgs fields of our monopoles grow linearly along the end, so the limit in Definition~\ref{definition: Mass and Charges of Monopoles} is infinite. We use instead the constant term of the asymptotic expansion, as for the Dirac monopoles in Subsection~\ref{subsection: Mass of $U(1)$-Dirac Monopoles}.

\begin{definition}[Mass and Charges on ACyl $3$-Manifolds]\label{definition: mass and charge ACyl}
    Let $(M,g)$ be an asymptotically cylindrical $3$-manifold with only one end $\Sigma\cong S^2$, and let $(A,\Phi)$ be a monopole on $M$ such that
    \begin{equation*}
        \lvert\Phi\rvert=m+c\rho+o(1),\qquad\text{as }\rho\to\infty,
    \end{equation*}
    for some constants $m\in\R$ and $c>0$. We call $m$ the \emph{mass} of $(A,\Phi)$; it depends on the choice of the distance function $\rho$ (Remark~\ref{remark: masses are determined}). The \emph{charge} of $(A,\Phi)$ is the degree of $\frac{\Phi}{\lvert\Phi\rvert}:\{\rho=R\}\to S^2$ for sufficiently large $R$, computed in a trivialization of $P$ over the end; it does not depend on the trivialization, since every map $S^2\to\SU(2)$ is null-homotopic.
\end{definition}

\begin{definition}[PS Monopoles]\label{definition:PS monopole}
    The \emph{PS monopole}, denoted by $(A_{PS},\Phi_{PS})$, is the unique (up to gauge equivalence) smooth $\SU(2)$-monopole of finite energy on $\R^3$ with mass $1$ and charge $1$ which is centered at the origin, i.e.\ whose Higgs field vanishes exactly at the origin. This special solution has an explicit expression
    \begin{equation*}
        A_{PS}(x)= ( \frac{1}{\sinh \lvert x\rvert}-\frac{1}{\lvert x\rvert} ) (\hat x \times \sigma)\cdot dx, \qquad
        \Phi_{PS}(x)= ( \frac{1}{\tanh \lvert x\rvert}-\frac{1}{\lvert x\rvert} ) (\hat x \cdot \sigma),
    \end{equation*}
    where $\hat x=\frac{x}{\lvert x\rvert}$ and $\sigma=\frac{1}{2} ( \sigma_1, \sigma_2, \sigma_3 ) \in \su(2) \otimes \R^3$.
\end{definition}

For $\lambda>0$, the scaled PS monopole with mass $\lambda$ is
\begin{equation*}
    (A_{PS}^\lambda(x),\Phi_{PS}^\lambda(x)):=(A_{PS}(\lambda x),\lambda \Phi_{PS}(\lambda x)),
\end{equation*}
where $A_{PS}(\lambda x)$ denotes the pull-back of $A_{PS}$ by $x\mapsto\lambda x$; it is a monopole by Proposition~\ref{proposition: Scaling Invariance}. On $\R^3\setminus\{0\}$, let
\begin{equation*}
    A_{WY}:=-\frac{1}{\lvert x\rvert}(\hat x\times\sigma)\cdot dx, \qquad \Phi_{WY}^\lambda:=(\lambda-\frac{1}{\lvert x\rvert})(\hat x\cdot\sigma)
\end{equation*}
be the Euclidean Dirac monopole with mass $\lambda$ and charge $1$; indeed, a direct computation shows that $\hat x\cdot\sigma$ is parallel with respect to $A_{WY}$ and that $*F_{A_{WY}}=d(-\frac{1}{\lvert x\rvert})(\hat x\cdot\sigma)=d_{A_{WY}}\Phi_{WY}^\lambda$. By Definition~\ref{definition:PS monopole},
\begin{equation}\label{equation: PS minus Dirac}
    \begin{aligned}
        A_{PS}^\lambda-A_{WY}&=\frac{\lambda}{\sinh(\lambda\lvert x\rvert)}(\hat x\times\sigma)\cdot dx,\\
        \Phi_{PS}^\lambda-\Phi_{WY}^\lambda&=\lambda(\frac{1}{\tanh(\lambda\lvert x\rvert)}-1)(\hat x\cdot\sigma),
    \end{aligned}
\end{equation}
and both are $O(\lambda e^{-\lambda\lvert x\rvert})$ for $\lambda\lvert x\rvert\geq1$. Thus, away from its core, the PS monopole is the Euclidean Dirac monopole up to exponentially small terms.

To glue PS monopoles at the singularities, we extend the bundle $P$ of Theorem~\ref{theorem:construction of Dirac monopoles on SU(2) bundle} over $p_1,\cdots,p_n$. Let $\eta_i' : B_{\varepsilon_0}(p_i)\to \R^3$ be oriented geodesic normal coordinates centered at $p_i$, where $\varepsilon_0$ is as in Subsection~\ref{subsection: Local Models of U(1)-Dirac Monopoles}.

\begin{definition}\label{definition:framing between R oplus L, su(2) over B(p)-p=R3-0}
    A \emph{framing} at $p_i$ is an isomorphism $\eta_i$ from $P|_{B_{\varepsilon_0}(p_i)\setminus\{p_i\}}$ to the trivial principal $\SU(2)$-bundle over $\eta_i'(B_{\varepsilon_0}(p_i))\setminus\{0\}$, covering $\eta_i'$, such that the induced isomorphism of adjoint bundles maps $\hat\sigma$ to $\hat x\cdot\sigma$. We write $\eta_i^*$ for the pull-back of connections and Higgs fields by $\eta_i$, and $\hat\sigma_i:=\eta_i^*(\hat x\cdot\sigma)=\hat\sigma$.
\end{definition}

\begin{lemma}\label{lemma:extending the principal SU(2)-bundle P to M-...}
    There exist framings $\eta_1,\cdots,\eta_n$ at $p_1,\cdots,p_n$. Gluing $P$ and the trivial bundles $B_{\varepsilon_0}(p_i)\times\SU(2)$ by them, we obtain a principal $\SU(2)$-bundle over $M$, which we again denote by $P$.
\end{lemma}
\begin{proof}
    Let $Q_0\to\R^3\setminus\{0\}$ be the principal $\U(1)$-bundle consisting of the pairs $(x,h)$ with $h\in\SU(2)$ and $\Ad_h(\frac{\sigma_3}{2})=\hat x\cdot\sigma$, on which $\U(1)\subset\SU(2)$ acts by right multiplication. The map $[(x,h),h']\mapsto(x,hh')$ identifies the associated principal $\SU(2)$-bundle with the trivial bundle, and it maps the section of the adjoint bundle corresponding to $\frac{\sigma_3}{2}$ to $\hat x\cdot\sigma$. Since $P$ is associated to $L'$ in the same way, every isomorphism of principal $\U(1)$-bundles from $L'|_{B_{\varepsilon_0}(p_i)\setminus\{p_i\}}$ to $(\eta_i')^*Q_0$ induces a framing at $p_i$. Since $B_{\varepsilon_0}(p_i)\setminus\{p_i\}$ retracts onto a $2$-sphere, these two principal $\U(1)$-bundles are isomorphic if and only if the curvatures of connections on them have the same integral over a small sphere around $p_i$. The curvature $*d\phi_D$ of $a_D$ has integral $4\pi k_i^I=4\pi$. On the other hand, $A_{WY}$ preserves $\hat x\cdot\sigma$, so it is induced by a connection on $Q_0$, whose curvature $*d(-\frac{1}{\lvert x\rvert})$ also has integral $4\pi$ over spheres around the origin. Hence the framings exist.
\end{proof}

The following lemma shows that, near $p_i$, the $\SU(2)$-Dirac monopole on $M$ is the Euclidean Dirac monopole up to terms of order $O(r_i)$.

\begin{lemma}\label{lemma: key estimation}
    Let $\lambda_i:=m_i^I+\overline{m}-1$. The framings $\eta_i$ can be chosen such that
    \begin{equation*}
        A_D^{\overline{m}}=\eta_i^*(A_{WY})+a_i\hat\sigma_i,\qquad \Phi_D^{\overline{m}}=\eta_i^*(\Phi_{WY}^{\lambda_i})+\psi_i\hat\sigma_i\qquad\text{on }B_{\varepsilon_0}(p_i)\setminus\{p_i\},
    \end{equation*}
    where $a_i$ is a real $1$-form and $\psi_i$ is a real function, both smooth on $B_{\varepsilon_0}(p_i)\setminus\{p_i\}$ and independent of $\overline{m}$, such that $\lvert a_i\rvert\leq Cr_i$ and $\lvert\nabla^k\psi_i\rvert\leq C_kr_i^{1-k}$ for all $k\geq0$, where the constants $C,C_k>0$ are independent of $\overline{m}$. In particular, $\phi_D^{\overline{m}}=\lambda_i-r_i^{-1}+\psi_i$.
\end{lemma}
\begin{proof}
    Let $g_0$ be the pull-back of the Euclidean metric by $\eta_i'$, and recall that $\lvert x\rvert=r_i$ in geodesic normal coordinates.
    \begin{enumerate}[label=\textbf{(\arabic*)},listparindent=\parindent]
        \item \textbf{The Higgs field.} Put $\psi_i:=\phi_D-m_i^I+r_i^{-1}$, which does not depend on $\overline{m}$. Since $\phi_D^{\overline{m}}=\phi_D+\overline{m}-1$, we have $\phi_D^{\overline{m}}=\lambda_i-r_i^{-1}+\psi_i$, i.e.\ $\Phi_D^{\overline{m}}=\eta_i^*(\Phi_{WY}^{\lambda_i})+\psi_i\hat\sigma_i$ for every framing $\eta_i$. In the proof of Theorem~\ref{proposition:existence of harmonic function PhiD}, we have $\psi_i=-c_i+v-v(p_i)$ with $v$ of class $C^1$, so $\lvert\psi_i\rvert\leq Cr_i$. Moreover, $\Delta\psi_i=\Delta r_i^{-1}=r_i^{-2}\partial_{r_i}\log\theta$ on $B_{\varepsilon_0}(p_i)\setminus\{p_i\}$, and since $\log\theta$ is smooth and vanishes to second order at $p_i$, we have $\lvert\nabla^k\Delta\psi_i\rvert\leq C_kr_i^{-1-k}$. Rescaling the annuli $\{s\leq r_i\leq 2s\}$ to unit size and applying interior estimates, we obtain $\lvert\nabla^k\psi_i\rvert\leq C_kr_i^{1-k}$.

        \item \textbf{The connection.} Since $A_D$ and $\eta_i^*(A_{WY})$ both preserve $\hat\sigma_i$, we have $A_D-\eta_i^*(A_{WY})=a_i'\hat\sigma_i$ for a smooth real $1$-form $a_i'$ on $B_{\varepsilon_0}(p_i)\setminus\{p_i\}$. Comparing the $\hat\sigma_i$-components of the curvatures, Theorem~\ref{theorem:construction of Dirac monopoles on SU(2) bundle} and $*_{g_0}F_{\eta_i^*(A_{WY})}=d(-r_i^{-1})\hat\sigma_i$ give
        \begin{equation*}
            da_i'=\omega_i:=*_gd\phi_D-*_{g_0}d(-r_i^{-1})=(*_g-*_{g_0})d(-r_i^{-1})+*_gd\psi_i .
        \end{equation*}
        Since $g-g_0$ vanishes to second order at $p_i$, we have $\lvert\nabla^k(*_g-*_{g_0})\rvert\leq C_kr_i^{\max\{2-k,0\}}$ (cf.\ \cite{esfahani_PhD} for $k=0$), and together with (1), we obtain $\lvert\nabla^k\omega_i\rvert\leq C_kr_i^{-k}$. In particular, $\omega_i$ is a bounded closed $2$-form on $B_{\varepsilon_0}(p_i)\setminus\{p_i\}$. Hence the radial homotopy formula in normal coordinates gives a primitive
        \begin{equation*}
            a_i(x):=\int_0^1t\,\iota_x\omega_i(tx)\,dt,\qquad da_i=\omega_i,\qquad\lvert a_i\rvert\leq Cr_i,
        \end{equation*}
        where the boundary term of the homotopy formula vanishes since $\omega_i$ is bounded, and $a_i$ is smooth on $B_{\varepsilon_0}(p_i)\setminus\{p_i\}$ by the bounds on $\nabla^k\omega_i$. Since $H^1(B_{\varepsilon_0}(p_i)\setminus\{p_i\};\R)=0$, we have $a_i'-a_i=d\vartheta_i$ for a smooth function $\vartheta_i$. Replacing $\eta_i$ by $h_i\circ\eta_i$, where the gauge transformation $h_i:=\exp(\vartheta_i\,\hat x\cdot\sigma)$ preserves $\hat x\cdot\sigma$, replaces $a_i'$ by $a_i'-d\vartheta_i=a_i$. All these quantities are independent of $\overline{m}$, since $A_D$ and $\phi_D$ are.\qedhere
    \end{enumerate}
\end{proof}

\subsection{Pre-gluing: Construction of Approximate Solutions}\label{subsection: Pre-gluing: Construction of Approximate Solutions}

In this subsection, we also assume $b^2(M)=0$, so that Lemma~\ref{lemma:large Higgs field} applies, and we fix the framings $\eta_i$ of Lemma~\ref{lemma: key estimation} and the principal $\SU(2)$-bundle $P\to M$ obtained from them as in Lemma~\ref{lemma:extending the principal SU(2)-bundle P to M-...}. Let $\lambda_i=m_i^I+\overline{m}-1$ be the scaling factors of Lemma~\ref{lemma: key estimation}, and put $\varepsilon_i:=\lambda_i^{-\frac12}=O(\overline{m}^{-\frac{1}{2}})$. If $\overline{m}$ is sufficiently large, then $\lambda_i>0$, $2\varepsilon_i<\varepsilon_0$, and $2\varepsilon_i>\varepsilon=\sqrt{\frac{5}{2\overline m}}$, so that $\phi_D^{\overline m}\geq\frac{\overline m}{2}$ outside $\bigsqcup_{i}B_{2\varepsilon_i}(p_i)$ by Lemma~\ref{lemma:large Higgs field}. The approximate solution $(A_{\mathrm{app}},\Phi_{\mathrm{app}})$ on $M$ consists of three parts in three disjoint regions:
\begin{itemize}
    \item On $M \setminus \bigsqcup_{i} B_{2\varepsilon_i}(p_i)$, it equals the scaled $\SU(2)$-Dirac monopole $(A_D^{\overline{m}},\Phi_D^{\overline{m}})$, with a large Higgs field $\lvert \Phi_{\mathrm{app}} \rvert = \phi_D^{\overline{m}} \geq \frac{\overline{m}}{2}$.
    \item On each $B_{\varepsilon_i}(p_i)$, using the geodesic normal coordinates $\eta_i'$, it equals the pull-back $(\eta^*_i(A_{PS}^{\lambda_i}),\eta^*_i(\Phi_{PS}^{\lambda_i}))$ of the scaled PS monopole on $\R^3$ with mass $\lambda_i$ by the framing $\eta_i$.
    \item On each annulus $B_{2\varepsilon_i}(p_i) \setminus B_{\varepsilon_i}(p_i)$, it interpolates between them, with the help of cut-off functions.
\end{itemize}
To glue these pull-back scaled PS monopoles onto the scaled $\SU(2)$-Dirac monopole, we choose smooth cut-off functions $\chi_i:M\to[0,1]$ for $i=1,\cdots,n$, with $\chi_i=1$ on $B_{\varepsilon_i}(p_i)$, $\supp\chi_i\subset B_{2\varepsilon_i}(p_i)$ and $\lvert d\chi_i\rvert\leq2\varepsilon_i^{-1}$, and put $\chi_0:=1-\sum_{i=1}^n\chi_i$. Therefore, our approximate solution on the whole $M$ is given by
\begin{align*}
    (A_{\mathrm{app}},\Phi_{\mathrm{app}}) := \chi_0(A_D^{\overline{m}},\Phi_D^{\overline{m}}) + \sum_{i=1}^{n} \chi_i(\eta^*_i(A_{PS}^{\lambda_i}),\eta^*_i(\Phi_{PS}^{\lambda_i})),
\end{align*}
which is smooth on $M$ but does not satisfy the Bogomolny equation.

\subsection{Error Term Estimation}\label{subsection: Error Term Estimation}

The error term $e_0=*F_{A_{\mathrm{app}}}-d_{A_{\mathrm{app}}}\Phi_{\mathrm{app}}$ vanishes on $M \setminus \bigsqcup_{i} B_{2\varepsilon_i}(p_i)$, since the scaled Dirac monopole $(A_D^{\overline{m}},\Phi_D^{\overline{m}})$ satisfies the Bogomolny equation there. On $\bigsqcup_{i} B_{2\varepsilon_i}(p_i)$, it comes from two sources: the scaled PS monopole satisfies the Bogomolny equation only with respect to the flat metric, while the metric $g$ on $B_{2\varepsilon_i}(p_i)$ is not flat in general, and the cut-off functions are not constant on the annuli $B_{2\varepsilon_i}(p_i)\setminus B_{\varepsilon_i}(p_i)$. The following lemma asserts that the error term is bounded independently of the average mass.

\begin{lemma}\label{lemma: error term estimation}
    If the average mass $\overline{m}$ is sufficiently large, then $e_0$ is supported in $\bigsqcup_{i=1}^n\overline{B_{2\varepsilon_i}(p_i)}$, and $\lvert e_0\rvert_g\le C$ for a constant $C>0$ independent of $\overline{m}$.
\end{lemma}

\begin{proof}
    On each $B_{2\varepsilon_i}(p_i)\setminus\{p_i\}$, put
    \begin{equation*}
        (A_1,\Phi_1):=(\eta^*_i (A_{PS}^{\lambda_i}),\eta^*_i (\Phi_{PS}^{\lambda_i})),\qquad(\alpha_i,\varphi_i):=(A_D^{\overline{m}}-A_1,\Phi_D^{\overline{m}}-\Phi_1).
    \end{equation*}
    Then the approximate solution can be written as
    \begin{equation*}
        (A_{\mathrm{app}},\Phi_{\mathrm{app}})=(A_1,\Phi_1)+\chi_0(\alpha_i,\varphi_i).
    \end{equation*}
    Since $F_{A+b}=F_A+d_Ab+\frac{1}{2}[b\wedge b]$ and $d_{A+b}(\Phi+\psi)=d_A\Phi+d_A\psi+[b,\Phi+\psi]$, we can express $F_{A_{\mathrm{app}}}$ and $d_{A_{\mathrm{app}}}\Phi_{\mathrm{app}}$, as well as $F_{A_D^{\overline{m}}}$ and $d_{A_D^{\overline{m}}}\Phi_D^{\overline{m}}$, in terms of $(A_1,\Phi_1)$ and $(\alpha_i,\varphi_i)$. Eliminating $d_{A_1}\alpha_i$ and $d_{A_1}\varphi_i+[\alpha_i,\Phi_1]$, we obtain
    \begin{equation}\label{equation: interpolation}
        \begin{aligned}
            F_{A_{\mathrm{app}}}&=(1-\chi_0)F_{A_1}+\chi_0F_{A_D^{\overline{m}}}+d\chi_0\wedge \alpha_i-\frac{1}{2}\chi_0(1-\chi_0)[\alpha_i\wedge \alpha_i],\\
            d_{A_{\mathrm{app}}}\Phi_{\mathrm{app}}&=(1-\chi_0)d_{A_1}\Phi_1+\chi_0d_{A_D^{\overline{m}}}\Phi_D^{\overline{m}}+\varphi_i d\chi_0-\chi_0(1-\chi_0)[\alpha_i,\varphi_i].
        \end{aligned}
    \end{equation}
    Since $(A_D^{\overline{m}},\Phi_D^{\overline{m}})$ satisfies the Bogomolny equation, the error term is
    \begin{gather*}
        e_0=(1-\chi_0)e_0^{(1)}+e_0^{(2)},\qquad e_0^{(1)}:=*_gF_{A_1}-d_{A_1}\Phi_1,\\
        e_0^{(2)}:=*_g(d\chi_0\wedge \alpha_i)-\varphi_i d\chi_0-\chi_0(1-\chi_0)Q(\alpha_i,\varphi_i),
    \end{gather*}
    where $Q(\alpha_i,\varphi_i)=*\frac{1}{2}[\alpha_i\wedge \alpha_i]-[\alpha_i,\varphi_i]$ is the quadratic term of Subsection~\ref{subsection: Deformation Complexes}.

    \begin{enumerate}[label=\textbf{(\arabic*)},listparindent=\parindent]
    \item \textbf{Estimation of $e_0^{(1)}$.} Let $g_0$ be the pull-back of the Euclidean metric by $\eta_i'$. Since $(A_1,\Phi_1)$ is a monopole with respect to $g_0$, we have $e_0^{(1)}=(*_g-*_{g_0})F_{A_1}$. We have $\lvert*_g-*_{g_0}\rvert\leq Cr_i^2$, and $\lvert F_{A_1}\rvert\le C\lambda_i^2(1+\lambda_i^2r_i^2)^{-1}$, since $\lvert F_{A_{PS}}\rvert(y)\le C(1+\lvert y\rvert^2)^{-1}$. Hence
    \begin{equation*}
        \lvert e_0^{(1)}\rvert_g\leq C\frac{r_i^2}{\lambda_i^{-2}+r_i^2}\leq C.
    \end{equation*}

    \item \textbf{Estimation of $e_0^{(2)}$.} Since $d\chi_0$ and $\chi_0(1-\chi_0)$ are supported in the annulus $B_{2\varepsilon_i}(p_i)\setminus B_{\varepsilon_i}(p_i)$, so is $e_0^{(2)}$. There, $\lambda_ir_i\geq\sqrt{\lambda_i}$ and $\lvert d\chi_0\rvert\le 2\varepsilon_i^{-1}$. By Lemma~\ref{lemma: key estimation} and \eqref{equation: PS minus Dirac},
    \begin{equation*}
        \alpha_i=a_i\hat\sigma_i-\frac{\lambda_i}{\sinh(\lambda_ir_i)}\eta_i^*((\hat x\times\sigma)\cdot dx),\qquad \varphi_i=(\psi_i-\lambda_i(\frac{1}{\tanh(\lambda_ir_i)}-1))\hat\sigma_i .
    \end{equation*}
    Here it is essential that the mass $\lambda_i$ of the scaled PS monopole equals the constant term $m_i^I+\overline{m}-1$ of $\phi_D^{\overline{m}}$ at $p_i$; otherwise $\varphi_i$ would contain a nonzero constant. Hence, if $\overline{m}$ is sufficiently large,
    \begin{equation}\label{equation: difference on the annulus}
        \lvert \alpha_i\rvert+\lvert\varphi_i\rvert\le C(r_i+\lambda_ie^{-\sqrt{\lambda_i}})\le C\varepsilon_i\qquad\text{on }B_{2\varepsilon_i}(p_i)\setminus B_{\varepsilon_i}(p_i).
    \end{equation}
    Since $\lvert Q(\alpha_i,\varphi_i)\rvert\le C(\lvert \alpha_i\rvert+\lvert\varphi_i\rvert)^2$, we obtain $\lvert e_0^{(2)}\rvert_g\le C$.\qedhere
    \end{enumerate}
\end{proof}

\section{Genuine Solutions}\label{section: Genuine Solutions}

In this section, we seek a perturbation $(a,\varphi)$ such that $(A_{\mathrm{app}}+a,\Phi_{\mathrm{app}}+\varphi)$ satisfies the Bogomolny equation. By \eqref{equation: not elliptic}, it must solve
\begin{equation}\label{equation: equation for the perturbation}
    d_2(a,\varphi)+Q((a,\varphi),(a,\varphi))+e_0=0 .
\end{equation}
The argument adapts the gluing constructions of \cite{esfahani_PhD} and \cite{Oliveira_2016}. In Subsection~\ref{subsection: Weighted Sobolev Spaces}, we introduce weighted Sobolev spaces $\mathcal H^0_\beta$, $\mathcal H^1_\beta$ and prove an elliptic estimate in these spaces. In Subsection~\ref{subsection: Solving Linear Equations}, we construct a right inverse $\mathcal G$ of $d_2$ whose norm is bounded independently of $\overline m$. In Subsection~\ref{subsection: Quadratic Terms and the Contraction Mapping Theorem}, we write $(a,\varphi)=\mathcal Gf$ and solve \eqref{equation: equation for the perturbation} by the contraction mapping theorem. Finally, in Subsection~\ref{subsection: Proof of the main theorem}, we show that the solution is gauge equivalent to a smooth irreducible $\SU(2)$-monopole with charge $n$ and mass $m=\overline{m}$.

\subsection{Weighted Sobolev Spaces}\label{subsection: Weighted Sobolev Spaces}

Throughout this section, $b^1(M)=b^2(M)=0$, and hence $M$ has a single end, with cross-section $\Sigma\cong S^2$ (Lemma~\ref{lemma: topology of the ends}). We use $C$ for constants independent of $\overline m$, allowing the value to change between occurrences. Since $m_i^I$ does not depend on $\overline m$, we have $\lvert\lambda_i-\overline m\rvert\le C$; in particular, $\frac12\overline m\le\lambda_i\le2\overline m$ once $\overline m$ is large.

Choose $r_0\in(0,\varepsilon_0]$, independently of $\overline m$, such that $r_0^2\lvert\Ric_g\rvert\le\frac1{36}$ on $B_{r_0}(p_i)$, and the Hardy inequality
\begin{equation}\label{equation: Hardy}
    \int_{B_{r_0}(p_i)}\frac{v^2}{r_i^2}\vol_g\le9\int_{B_{r_0}(p_i)}\lvert dv\rvert^2\vol_g
\end{equation}
holds for every compactly supported Lipschitz function $v$ on $B_{r_0}(p_i)$. Such an $r_0$ exists by the Euclidean Hardy inequality (with constant $4$) and the closeness of $g$ to the Euclidean metric in normal coordinates.

To record the scale near each singular point, choose $r_*:M\to[0,1]$, independently of $\overline m$ and smooth on $M\setminus S_p$, with $r_*=r_i$ on $B_{\frac{r_0}{2}}(p_i)$, $r_*\ge\frac{r_0}{2}$ outside $\bigsqcup\limits_{i=1}^nB_{\frac{r_0}{2}}(p_i)$, $r_*=1$ outside $\bigsqcup\limits_{i=1}^nB_{r_0}(p_i)$, and $\lvert\nabla r_*\rvert\le C$. We put
\begin{equation}\label{equation: weight w}
    w:=\sqrt{\overline m^{-2}+r_*^2}.
\end{equation}
Then $w$ is smooth, $\overline m^{-1}\le w\le\sqrt2$, $\lvert\nabla w\rvert\le C$, and $r_*\le w\le\sqrt2r_*$ wherever $r_*\ge\overline m^{-1}$. We fix $\beta\in(\beta^*,0)$ as in Theorem~\ref{theorem: main theorem}.

\begin{definition}\label{definition: weighted Sobolev spaces}
    Let $U\subset M$ be open. For all $f \in \Omega^1(U;\mathfrak{g}_P)$ and $\xi=(a,\varphi) \in \Omega^1(U;\mathfrak g_P)\oplus\Omega^0(U;\mathfrak g_P)$, we define
    \begin{align*}
        \Vert f\Vert_{\mathcal H^0_\beta(U)}&:=\big\Vert we^{-\frac12\beta\rho}f\big\Vert_{L^2(U)},\\
        \Vert\xi\Vert^2_{\mathcal H^1_\beta(U)}&:=\big\Vert e^{-\frac12\beta\rho}\xi\big\Vert^2_{L^2(U)}
        +\big\Vert we^{-\frac12\beta\rho}\nabla_{A_{\mathrm{app}}}\xi\big\Vert^2_{L^2(U)}
        +\big\Vert w e^{-\frac12\beta\rho}[\Phi_{\mathrm{app}},\xi]\big\Vert^2_{L^2(U)}.
    \end{align*}
    The weighted Sobolev spaces $\mathcal H^0_\beta(U)$ and $\mathcal H^1_\beta(U)$ consist of all $f$ and $\xi$, respectively, whose norms are finite. We omit $U$ if $U=M$.
\end{definition}

For fixed $\overline m$, we have $w\geq\overline m^{-1}>0$, so these norms are locally equivalent to the usual $L^2$- and $L^2_1$-norms, and $\mathcal H^0_\beta(U)$ and $\mathcal H^1_\beta(U)$ are Hilbert spaces.

\begin{lemma}\label{lemma: weighted Sobolev}
    For all $\xi,\xi'\in\mathcal H^1_\beta$, the following estimates hold.
    \begin{enumerate}[label=\textbf{(\arabic*)},listparindent=\parindent]
        \item $\Vert d_2\xi\Vert_{\mathcal H^0_\beta}\le C\Vert\xi\Vert_{\mathcal H^1_\beta}$ and $\Vert we^{-\frac12\beta\rho}\xi\Vert_{L^6}\le C\Vert\xi\Vert_{\mathcal H^1_\beta}$.
        \item $\big\Vert we^{-\frac12\beta\rho}\lvert\xi\rvert\lvert\xi'\rvert\big\Vert_{L^2}\le C\overline m^{\frac{1}{2}}\Vert\xi\Vert_{\mathcal H^1_\beta}\Vert\xi'\Vert_{\mathcal H^1_\beta}$.
    \end{enumerate}
\end{lemma}
\begin{proof}
    The first estimate in (1) follows from the definition of $d_2$. For the second, recall that $(M,g)$ has bounded geometry, so the Sobolev inequality $\Vert u\Vert_{L^6}\le C(\Vert du\Vert_{L^2}+\Vert u\Vert_{L^2})$ holds for functions $u$ on $M$. We apply it to $u=w e^{-\frac12\beta\rho}\lvert\xi\rvert$, for which Kato's inequality gives
    $\lvert du\rvert\le w e^{-\frac12\beta\rho}\lvert\nabla_{A_{\mathrm{app}}}\xi\rvert+C e^{-\frac12\beta\rho}\lvert\xi\rvert$.

    For (2), H\"older's inequality and (1) give
    \begin{equation*}
        \big\Vert w^{\frac{3}{4}}e^{-\frac12\beta\rho}\xi\big\Vert_{L^4}^4\le\Vert e^{-\frac12\beta\rho}\xi\Vert_{L^2}\Vert w e^{-\frac12\beta\rho}\xi\Vert_{L^6}^3\le C\Vert\xi\Vert_{\mathcal H^1_\beta}^4,
    \end{equation*}
    and likewise for $\xi'$. Since $w^{-\frac{1}{2}}e^{\frac12\beta\rho}\le\overline m^{\frac{1}{2}}$, we have
    \begin{equation*}
        w e^{-\frac12\beta\rho}\lvert\xi\rvert\lvert\xi'\rvert\le\overline m^{\frac{1}{2}}\big(w^{\frac{3}{4}}e^{-\frac12\beta\rho}\lvert\xi\rvert\big)\big(w^{\frac{3}{4}}e^{-\frac12\beta\rho}\lvert\xi'\rvert\big),
    \end{equation*}
    and the Cauchy--Schwarz inequality gives the result.
\end{proof}

The weight $w$ is adapted to the approximate solution in the following sense.

\begin{lemma}\label{lemma: approximate solution}
    We have
    \begin{equation}\label{equation: curvature of app}
        \lvert\Ric_g\rvert+\lvert F_{A_{\mathrm{app}}}\rvert+\lvert d_{A_{\mathrm{app}}}\Phi_{\mathrm{app}}\rvert\le Cw^{-2}.
    \end{equation}
\end{lemma}
\begin{proof}
    We have $\lvert F_{A_{\mathrm{app}}}\rvert\le\lvert d_{A_{\mathrm{app}}}\Phi_{\mathrm{app}}\rvert+\lvert e_0\rvert$, while $\lvert\Ric_g\rvert$ and $\lvert e_0\rvert$ are bounded by the asymptotic cylindricity of $g$ and Lemma~\ref{lemma: error term estimation}. Moreover, $w\le\sqrt2$ on $M$, and $\lambda_i^{-2}+r_i^2\ge cw^2$ on $B_{r_0}(p_i)$ for some $c>0$, since $\lambda_i\le2\overline m$. Hence it suffices to show that $d_{A_{\mathrm{app}}}\Phi_{\mathrm{app}}$ is bounded outside $\bigsqcup\limits_{i=1}^nB_{r_0}(p_i)$, and that
    \begin{equation*}
        \lvert d_{A_{\mathrm{app}}}\Phi_{\mathrm{app}}\rvert\leq C(\lambda_i^{-2}+r_i^2)^{-1}\qquad\text{on }B_{r_0}(p_i).
    \end{equation*}
    Outside $\bigsqcup\limits_{i=1}^nB_{2\varepsilon_i}(p_i)$, we have $\lvert d_{A_{\mathrm{app}}}\Phi_{\mathrm{app}}\rvert=\lvert d\phi_D^{\overline{m}}\rvert$, since the approximate solution is the scaled $\SU(2)$-Dirac monopole there. This is bounded outside $\bigsqcup\limits_{i=1}^nB_{r_0}(p_i)$ by Theorem~\ref{proposition:existence of harmonic function PhiD}, and it is bounded by $Cr_i^{-2}\le2C(\lambda_i^{-2}+r_i^2)^{-1}$ on $B_{r_0}(p_i)\setminus B_{2\varepsilon_i}(p_i)$, since $\lvert d(\phi_D^{\overline{m}}+r_i^{-1})\rvert=\lvert d\psi_i\rvert\le C$ by Lemma~\ref{lemma: key estimation}. On $B_{2\varepsilon_i}(p_i)$, we use \eqref{equation: interpolation}, where $\chi_0=0$ on $B_{\varepsilon_i}(p_i)$. Since $\lvert d_{A_{PS}}\Phi_{PS}\rvert(y)\le C(1+\lvert y\rvert^2)^{-1}$, the first term is bounded by $C(\lambda_i^{-2}+r_i^2)^{-1}$, and by \eqref{equation: difference on the annulus}, the other terms are bounded by $Cr_i^{-2}+C$ on $B_{2\varepsilon_i}(p_i)\setminus B_{\varepsilon_i}(p_i)$.
\end{proof}

Since $d_2$ is not elliptic, we estimate solutions of $d_2\xi=f$ by the elliptic operator $D=d_2\oplus d_1^*$. For a pair $\xi=(a,\varphi)$ and a function $\gamma$, a direct computation gives
\begin{equation}\label{equation: commutator}
    D(\gamma\xi)-\gamma D\xi=(*(d\gamma\wedge a)-\varphi d\gamma,\langle d\gamma,a\rangle),\qquad\lvert D(\gamma\xi)-\gamma D\xi\rvert=\lvert\nabla\gamma\rvert\lvert\xi\rvert,
\end{equation}
and the first component of $D(\gamma\xi)-\gamma D\xi$ is $d_2(\gamma\xi)-\gamma d_2\xi$.

\begin{lemma}\label{lemma: interior estimate}
    For every pair $\xi$ on $M$ of class $L^2_{1,loc}$,
    \begin{equation*}
        \Vert\xi\Vert_{\mathcal H^1_\beta}\le C(\Vert e^{-\frac12\beta\rho}\xi\Vert_{L^2}+\Vert we^{-\frac12\beta\rho}D\xi\Vert_{L^2}).
    \end{equation*}
\end{lemma}
\begin{proof}
    Set $N=\Vert e^{-\frac12\beta\rho}\xi\Vert_{L^2}+\Vert we^{-\frac12\beta\rho}D\xi\Vert_{L^2}$. We may assume that $\xi$ has compact support: in general, we apply this case to $\eta_S\xi$, where $\eta_S=\min\{1,\max\{0,2-\frac{\rho}{S}\}\}$, and let $S\to\infty$, using \eqref{equation: commutator} and $w\lvert\nabla\eta_S\rvert\le1$.

    Put $\vartheta=we^{-\frac12\beta\rho}\xi$. Since $\lvert\nabla w\rvert\le C$, $w\le\sqrt2$ and $\lvert\nabla e^{-\frac12\beta\rho}\rvert\le Ce^{-\frac12\beta\rho}$, we have $\lvert\nabla(we^{-\frac12\beta\rho})\rvert\le Ce^{-\frac12\beta\rho}$, so \eqref{equation: commutator} gives $\Vert D\vartheta\Vert_{L^2}\le\Vert we^{-\frac12\beta\rho} D\xi\Vert_{L^2}+C\Vert e^{-\frac12\beta\rho}\xi\Vert_{L^2}\le CN$. By Lemma~\ref{lemma: Weitzenbock D} and \eqref{equation: curvature of app}, $\lvert\langle\mathcal Z\vartheta,\vartheta\rangle_{L^2}\rvert\le C\Vert w^{-1}\vartheta\Vert_{L^2}^2\le CN^2$. Hence \eqref{equation: Weitzenbock D} gives
    \begin{equation*}
        \Vert\nabla_{A_{\mathrm{app}}}\vartheta\Vert_{L^2}+\Vert[\Phi_{\mathrm{app}},\vartheta]\Vert_{L^2}\le CN.
    \end{equation*}
    This proves the lemma, since
    \begin{equation*}
        w e^{-\frac12\beta\rho}\nabla_{A_{\mathrm{app}}}\xi=\nabla_{A_{\mathrm{app}}}\vartheta-d(w e^{-\frac12\beta\rho})\otimes\xi,
        \qquad w e^{-\frac12\beta\rho}[\Phi_{\mathrm{app}},\xi]=[\Phi_{\mathrm{app}},\vartheta].\qedhere
    \end{equation*}
\end{proof}

\subsection{Solving Linear Equations}\label{subsection: Solving Linear Equations}

In this subsection, we construct a right inverse for $d_2$ with norm bounded independently of $\overline m$.

\begin{theorem}\label{theorem: surjective}
    If $\overline m$ is sufficiently large, then there is a bounded linear map $\mathcal G:\mathcal H^0_\beta\to\mathcal H^1_\beta$ such that
    \begin{equation*}
        d_2\circ\mathcal G=\mathrm{Id},\qquad\Vert\mathcal Gf\Vert_{\mathcal H^1_\beta}\le C_G\Vert f\Vert_{\mathcal H^0_\beta},
    \end{equation*}
    where $C_G$ is independent of $\overline m$. In particular, $d_2:\mathcal H^1_\beta\to\mathcal H^0_\beta$ is surjective.
\end{theorem}

We fix $t\in(0,\frac{r_0}{12}]$, independently of $\overline m$, put
\begin{equation*}
    M_{\mathrm{ext}}:=M\setminus\overline{\bigsqcup_{i=1}^n{B_t(p_i)}},
\end{equation*}
and prove Theorem~\ref{theorem: surjective} in three steps (Figure~\ref{figure: radii}).
\begin{itemize}
    \item The first step is to solve $d_2\xi=f$ on each ball $B_{r_0}(p_i)$, where the Hardy inequality \eqref{equation: Hardy} makes $d_2d_2^*$ coercive (Subsubsection~\ref{subsubsection: ball problem}).
    \item The second step is to solve $D\xi=(f,0)$ on $M_{\mathrm{ext}}$ for $f$ vanishing near $p_1,\cdots,p_n$. There, the approximate solution is reducible, and $D$ splits into an abelian part and a part with a large Higgs term (Subsubsection~\ref{subsection: Dirac monopole and approximate solution}).
    \item The third step is to patch the local solutions by cut-off functions. The patching error is small, since the exterior solution is small near $p_i$, and it is removed by a Neumann series (Subsubsection~\ref{subsubsection: Patching}).
\end{itemize}
The constants $C$ below are independent of $t$, while the lower bound on $\overline m$ may depend on $t$; we fix $t$ in the third step.

\begin{figure}[ht]
    \centering
    \begin{tikzpicture}[x=1.02cm,y=0.72cm,font=\small]
        \draw[->] (0,0) -- (12.4,0) node[right] {$r_i$};
        \foreach \x/\l in {0/0,1.2/2\varepsilon_i,2.4/t,3.6/2t,4.8/3t,7.6/\frac{r_0}{2},11/r_0}
            \draw (\x,0.1) -- (\x,-0.1) node[below] {$\l$};
        \draw[|-|] (0,1) -- node[above] {$\supp e_0$} (1.2,1);
        \draw[|-|] (3.6,1) -- node[above] {$\gamma_0\colon0\to1$} (4.8,1);
        \draw[|-|] (7.6,1) -- node[above] {$\gamma_i\colon1\to0$} (11,1);
        \draw[|-|] (0,2.2) -- node[above] {$f_0=0$} (7.6,2.2);
        \draw[|->] (2.4,3.4) -- node[above] {$M_{\mathrm{ext}}$: $D\xi^{(0)}=(f_0,0)$} (12.2,3.4);
        \draw[|-|] (0,4.6) -- node[above] {$B_{r_0}(p_i)$: $d_2\xi^{(i)}=f_i$} (11,4.6);
    \end{tikzpicture}
    \caption{The radii near $p_i$ and the supports of the cut-off functions.}
    \label{figure: radii}
\end{figure}

\subsubsection{Near the Singularities}\label{subsubsection: ball problem}

\begin{proposition}\label{proposition: ball problem}
    If $\overline m$ is sufficiently large, then for every $f\in\mathcal H^0_\beta(B_{r_0}(p_i))$ there is a pair $\xi$ on $B_{r_0}(p_i)$ of class $L^2_{1,loc}$, depending linearly on $f$, such that $d_2\xi=f$ on $B_{r_0}(p_i)$ and
    \begin{equation*}
        \Vert e^{-\frac12\beta\rho}\xi\Vert_{L^2(B_{r_0}(p_i))}+\Vert we^{-\frac12\beta\rho}D\xi\Vert_{L^2(B_{r_0}(p_i))}\le C\Vert f\Vert_{\mathcal H^0_\beta(B_{r_0}(p_i))} .
    \end{equation*}
\end{proposition}
\begin{proof}
    Since $B_{r_0}(p_i)\subset\{\rho<R_0\}$, the weight $e^{-\frac12\beta\rho}$ is bounded above and below on $B_{r_0}(p_i)$, so we omit it. Let $\mathcal E_i$ be the completion of the space of compactly supported smooth $\mathfrak g_P$-valued $1$-forms on $B_{r_0}(p_i)$ with respect to
    \begin{equation*}
        \Vert u\Vert_{\mathcal E}^2:=\Vert\nabla_{A_{\mathrm{app}}}u\Vert^2_{L^2}+\Vert[\Phi_{\mathrm{app}},u]\Vert^2_{L^2}.
    \end{equation*}
    By Kato's inequality and \eqref{equation: Hardy}, $\Vert r_i^{-1}u\Vert_{L^2}\le3\Vert u\Vert_{\mathcal E}$ for $u\in\mathcal E_i$.

    We first show that $d_2d_2^*$ is coercive on $\mathcal E_i$. By the Weitzenb\"ock formula~\eqref{equation:d2d2*=Weitzenbock},
    \begin{equation*}
        \Vert d_2^*u\Vert^2_{L^2}=\Vert u\Vert_{\mathcal E}^2+\langle\Ric(u)+*[e_0\wedge u],u\rangle_{L^2}.
    \end{equation*}
    Since $r_0^2\lvert\Ric_g\rvert\le\frac1{36}$, and $e_0$ is supported in $B_{2\varepsilon_i}(p_i)$ with $\lvert e_0\rvert\le C$, the last term is bounded by $(\frac1{36}+C\varepsilon_i^2)\Vert r_i^{-1}u\Vert^2_{L^2}\le\frac12\Vert u\Vert_{\mathcal E}^2$ if $\overline m$ is large. Hence $\Vert d_2^*u\Vert^2_{L^2}\ge\frac12\Vert u\Vert_{\mathcal E}^2$.

    Next, since $w\ge r_*\ge\frac12r_i$ on $B_{r_0}(p_i)$, we have $\Vert w^{-1}v\Vert_{L^2}\le6\Vert v\Vert_{\mathcal E}$, and hence $\lvert\langle f,v\rangle_{L^2}\rvert\le6\Vert wf\Vert_{L^2}\Vert v\Vert_{\mathcal E}$ for $v\in\mathcal E_i$. Thus the Lax--Milgram theorem gives a unique $u\in\mathcal E_i$, depending linearly on $f$, such that
    \begin{equation*}
        \langle d_2^*u,d_2^*v\rangle_{L^2}=\langle f,v\rangle_{L^2}\quad\text{for all }v\in\mathcal E_i,
        \qquad \Vert u\Vert_{\mathcal E}\le12\Vert wf\Vert_{L^2}.
    \end{equation*}
    Then $\xi:=d_2^*u$ satisfies $d_2\xi=f$ in the sense of distributions, and $\Vert\xi\Vert^2_{L^2}=\langle f,u\rangle_{L^2}\le C\Vert wf\Vert^2_{L^2}$.

    Finally, we estimate $D\xi$. By Lemma~\ref{lemma: Weitzenbock D} with $e=e_0$, we have $D\xi=(f,e_0\star u)$. Since $w\lvert e_0\rvert\le C$,
    \begin{equation*}
        \Vert we_0\star u\Vert_{L^2}\le C\Vert u\Vert_{L^2}\le C\Vert r_i^{-1}u\Vert_{L^2}\le C\Vert wf\Vert_{L^2}.
    \end{equation*}
    Moreover, $\xi\in L^2_{1,loc}$ by elliptic regularity for $D$.
\end{proof}

\subsubsection{On the Exterior Region}\label{subsection: Dirac monopole and approximate solution}

If $\overline m$ is large, then $2\varepsilon_i<t$ and $\sqrt{\frac{5}{2\overline m}}<t$. Hence, by Subsection~\ref{subsection: Pre-gluing: Construction of Approximate Solutions} and Lemma~\ref{lemma:large Higgs field}, the approximate solution on $M_{\mathrm{ext}}$ is the scaled $\SU(2)$-Dirac monopole: $e_0=0$, and $\Phi_{\mathrm{app}}=\phi_D^{\overline m}\hat\sigma$ with $\phi_D^{\overline m}\ge\frac{\overline m}2$, where $\hat\sigma$ is parallel by Theorem~\ref{theorem:construction of Dirac monopoles on SU(2) bundle}. On $M_{\mathrm{ext}}$, the adjoint bundle splits orthogonally as $\mathfrak g_P=\underline\R\oplus L$, where $\underline\R$ is spanned by $\hat\sigma$ and $L=\hat\sigma^\perp$, and one writes a section $f$ of $\Lambda^k\otimes\mathfrak g_P$ as $f=f^\parallel\oplus f^\perp$ according to this splitting. These components will be called the longitudinal and the transverse component, respectively, and we identify the longitudinal component $f^\parallel$ with a real $k$-form by means of $\hat\sigma$. Since $\hat\sigma$ is parallel and commutes with $\Phi_{\mathrm{app}}$, the operators $D$, $d_2$ and $d_2^*$ preserve this splitting. We solve the longitudinal and transverse parts separately, and then combine them.

\paragraph{The longitudinal part.}
On longitudinal pairs, $D$ is abelian: since $\hat\sigma$ is parallel and $[\Phi_{\mathrm{app}},\hat\sigma]=0$, we have
\begin{equation}\label{equation: parallel identities}
    D(\zeta\oplus0)=D_0\zeta\oplus0\qquad\text{on }M_{\mathrm{ext}}
\end{equation}
for every pair $\zeta=(a,\varphi)$ of real forms,
where $D_0$ is the operator $D$ for the trivial pair on the trivial real line bundle, i.e.
\begin{align}
    D_0:\Omega^1(M)\oplus\Omega^0(M)&\to\Omega^1(M)\oplus\Omega^0(M)\notag\\
    (a,\varphi)&\mapsto(*da-d\varphi,-d^*a).\label{equation: D0}
\end{align}
By the proof of Lemma~\ref{lemma: Weitzenbock D} for the trivial pair, $D_0=S\circ(d+d^*)\circ T$. Hence $D_0$ is an asymptotically cylindrical elliptic operator, formally self-adjoint, and $D_0^2=(dd^*+d^*d)\oplus d^*d$ is the Hodge Laplacian on $\Omega^1(M)\oplus\Omega^0(M)$. We use the spaces $L^2_{k,\pm\beta}$ of Definition~\ref{definition: Lockhart-McOwen Sobolev spaces}, where $\Vert\zeta\Vert_{L^2_{0,\beta}}=\Vert e^{-\frac12\beta\rho}\zeta\Vert_{L^2}$.

\begin{proposition}\label{proposition: longitudinal}
    Let $f\in L^2_{0,\beta}(M)$ be a real-valued $1$-form. Then there is a unique real-valued pair $\zeta=(a,\varphi)\in L^2_{1,\beta}(M)$ such that $D_0\zeta=(f,0)$, i.e.\ $*da-d\varphi=f$ and $d^*a=0$. Moreover, $\Vert\zeta\Vert_{L^2_{1,\beta}}\le C\Vert f\Vert_{L^2_{0,\beta}}$.
\end{proposition}
\begin{proof}
    We first show that $D_0:L^2_{1,\pm\beta}\to L^2_{0,\pm\beta}$ are Fredholm, and compute their indices. On the end, write $a=a_\rho d\rho+a_\Sigma$. If a finite sum of sections of the form $e^{(\gamma+i\nu)\rho}\rho^ju_{\nu,j}$ as in Definition~\ref{definition: Fredholm} solves the translation-invariant model of $D_0$, then it solves the model $-\partial_\rho^2+\Delta_\Sigma$ of $D_0^2$, which acts on each of $a_\rho$, $a_\Sigma$ and $\varphi$. Since $\Sigma\cong S^2$, there are no harmonic $1$-forms on $\Sigma$, and the nonzero eigenvalues of the Laplacian on $1$-forms on $\Sigma$ are those of the Laplacian on functions. Separating variables, we see that the critical rates of $D_0$ satisfy $\gamma=0$ or $\lvert\gamma\rvert\ge\sqrt{\mu_1}$, and that for $\gamma=0$, we have $a_\Sigma=0$, while $a_\rho$ and $\varphi$ are of the form $c+c'\rho$ with constants $c,c'$. On such pairs, the model of $D_0$ is $(a_\rho d\rho,\varphi)\mapsto(-\partial_\rho\varphi\, d\rho,\partial_\rho a_\rho)$, so the solutions for $\gamma=0$ are the constant pairs $(c_1d\rho,c_2)$, and $d(0)=2$. Since $0<\frac12\lvert\beta\rvert<\sqrt{\mu_1}$, Theorem~\ref{theorem: Fredholm operator} shows that both operators are Fredholm, and that the index of $D_0$ on $L^2_{1,-\beta}$ exceeds its index on $L^2_{1,\beta}$ by $2$. Since $D_0$ is formally self-adjoint, the cokernel of each of them is the kernel of the other one. Hence the index of $D_0$ on $L^2_{1,\beta}$ is $-1$.

    We next show that $D_0$ is injective on $L^2_{1,\beta}$. Let $\zeta=(a,\varphi)\in L^2_{1,\beta}$ satisfy $D_0\zeta=0$. By elliptic regularity, $\zeta\in L^2_{k,\beta}$ for all $k$, so $\zeta$ is smooth and $\lvert\zeta\rvert\le Ce^{\frac12\beta\rho}$. Since $D_0^2$ is the Hodge Laplacian, $\varphi$ is harmonic, and $\varphi=0$ by Lemma~\ref{lemma: d*d(L(M)) = {intM u vol=0}}(1). Thus $da=0$ and $d^*a=0$. Since $b^1(M)=0$, we have $a=dh$ for a harmonic function $h$, which is bounded, since $\lvert dh\rvert\le Ce^{\frac12\beta\rho}$. By Lemma~\ref{lemma: d*d(L(M)) = {intM u vol=0}}(2), $h$ is constant, and hence $a=0$.

    Consequently, the cokernel of $D_0$ on $L^2_{1,\beta}$ is one-dimensional. It is the kernel of $D_0$ on $L^2_{1,-\beta}$, which contains the constant pairs $(0,c)$, so it consists of them. Hence the image of $D_0:L^2_{1,\beta}\to L^2_{0,\beta}$ consists of the pairs $(f,f_0)$ with $\int_Mf_0\vol_g=0$. In particular, it contains $(f,0)$, and the open mapping theorem gives $\zeta$ and the estimate.
\end{proof}

\paragraph{The transverse part.}
For every $L$-valued form $u$, we have $\lvert[\Phi_{\mathrm{app}},u]\rvert=\lvert\Phi_{\mathrm{app}}\rvert\lvert u\rvert\ge\frac{\overline m}2\lvert u\rvert$ on $M_{\mathrm{ext}}$ by \eqref{equation:|[X3s3,Y1s1+Y2s2]|=|X3s3||Y1s1+Y2s2|}. In the rest of this paragraph, all norms are taken on $M_{\mathrm{ext}}$. Let $\mathcal E_{\mathrm{ext}}$ be the completion of the space of compactly supported smooth $L$-valued $1$-forms on $M_{\mathrm{ext}}$ with respect to
\begin{equation*}
    \Vert u\Vert^2_{\mathcal E_{\mathrm{ext}}}:=\Vert\nabla_{A_{\mathrm{app}}}u\Vert^2_{L^2}+\Vert\lvert\Phi_{\mathrm{app}}\rvert u\Vert^2_{L^2}.
\end{equation*}
Then $\Vert u\Vert_{L^2}\le2\overline m^{-1}\Vert\lvert\Phi_{\mathrm{app}}\rvert u\Vert_{L^2}$ for $u\in\mathcal E_{\mathrm{ext}}$.

\begin{proposition}\label{proposition: transverse}
    If $\overline m$ is sufficiently large, then for every $L$-valued $1$-form $f^\perp$ on $M_{\mathrm{ext}}$ with $\Vert e^{-\frac12\beta\rho}f^\perp\Vert_{L^2}<\infty$, there is a unique $u\in\mathcal E_{\mathrm{ext}}$ such that
    \begin{equation}\label{equation: weak equation}
        \langle d_2^*u,d_2^*v\rangle_{L^2}=\langle f^\perp,v\rangle_{L^2}\qquad\text{for all }v\in\mathcal E_{\mathrm{ext}} .
    \end{equation}
    Moreover, $d_2d_2^*u=f^\perp$, and
    \begin{equation*}
        \Vert e^{-\frac12\beta\rho}d_2^*u\Vert_{L^2}\le C\overline m^{-1}\Vert e^{-\frac12\beta\rho}f^\perp\Vert_{L^2}.
    \end{equation*}
\end{proposition}
\begin{proof}
    Since $e_0=0$ on $M_{\mathrm{ext}}$, and since $-[\Phi_{\mathrm{app}},[\Phi_{\mathrm{app}},u]]=\lvert\Phi_{\mathrm{app}}\rvert^2u$ for $L$-valued $u$, the Weitzenb\"ock formula~\eqref{equation:d2d2*=Weitzenbock} gives
    \begin{equation*}
        \Vert d_2^*u\Vert_{L^2}^2=\Vert u\Vert^2_{\mathcal E_{\mathrm{ext}}}+\langle\Ric(u),u\rangle_{L^2}\qquad\text{for all }u\in\mathcal E_{\mathrm{ext}}.
    \end{equation*}
    Since $\lvert\Ric_g\rvert\le C$ and $\lvert\Phi_{\mathrm{app}}\rvert\ge\frac{\overline m}2$, the last term is absorbed if $\overline m$ is large, and
    \begin{equation}\label{equation: coercivity}
        \Vert d_2^*u\Vert_{L^2}^2\ge\Vert\nabla_{A_{\mathrm{app}}}u\Vert_{L^2}^2+\tfrac12\Vert\lvert\Phi_{\mathrm{app}}\rvert u\Vert_{L^2}^2\qquad\text{for all }u\in\mathcal E_{\mathrm{ext}}.
    \end{equation}
    Since $e^{\frac12\beta\rho}\le1$, we have $\lvert\langle f^\perp,v\rangle_{L^2}\rvert\le2\overline m^{-1}\Vert e^{-\frac12\beta\rho}f^\perp\Vert_{L^2}\Vert\lvert\Phi_{\mathrm{app}}\rvert v\Vert_{L^2}$, so the Lax--Milgram theorem gives $u$. Since $d_2^*$ preserves the splitting, both sides of \eqref{equation: weak equation} vanish for longitudinal test forms $v$. Hence $d_2d_2^*u=f^\perp$ in the sense of distributions, and $u\in L^2_{2,loc}$ by elliptic regularity.

    Put $F=\Vert e^{-\frac12\beta\rho}f^\perp\Vert_{L^2}$. For $R>0$, let $\omega_R=e^{-\frac12\beta\min\{\rho,R\}}$ and $X=\Vert\lvert\Phi_{\mathrm{app}}\rvert\omega_Ru\Vert_{L^2}$, so that $\Vert\omega_Ru\Vert_{L^2}\le2\overline m^{-1}X$. Then $\lvert\nabla\omega_R\rvert\le C\omega_R$, and $\omega_Ru,\omega_R^2u\in\mathcal E_{\mathrm{ext}}$. Since $d_2^*(\omega_Ru)-\omega_Rd_2^*u$ is algebraic in $u$ and bounded by $C\omega_R\lvert u\rvert$, and since $\omega_R\le e^{-\frac12\beta\rho}$, testing \eqref{equation: weak equation} with $v=\omega_R^2u$ gives
    \begin{align*}
        \Vert d_2^*(\omega_Ru)\Vert_{L^2}^2
        &=\langle f^\perp,\omega_R^2u\rangle_{L^2}+\Vert d_2^*(\omega_Ru)-\omega_Rd_2^*u\Vert_{L^2}^2\\
        &\le2\overline m^{-1}FX+C\overline m^{-2}X^2.
    \end{align*}
    By \eqref{equation: coercivity}, the left hand side is at least $\frac12X^2$. Hence $X\le C\overline m^{-1}F$ if $\overline m$ is large, and then $\Vert\omega_Rd_2^*u\Vert_{L^2}\le\Vert d_2^*(\omega_Ru)\Vert_{L^2}+C\Vert\omega_Ru\Vert_{L^2}\le C\overline m^{-1}F$. Letting $R\to\infty$ proves the estimate.
\end{proof}

We now combine the two parts.

\begin{proposition}\label{proposition: exterior problem}
    If $\overline m$ is sufficiently large, then for every $f\in\mathcal H^0_\beta$ which vanishes on $\bigsqcup\limits_{i=1}^nB_{\frac{r_0}{2}}(p_i)$, there is a pair $\xi$ on $M_{\mathrm{ext}}$ of class $L^2_{1,loc}$, depending linearly on $f$, such that $D\xi=(f,0)$ on $M_{\mathrm{ext}}$ and
    \begin{align*}
        \Vert e^{-\frac12\beta\rho}\xi\Vert_{L^2(M_{\mathrm{ext}})}&\le C\Vert f\Vert_{\mathcal H^0_\beta},\\
        \Vert e^{-\frac12\beta\rho}\xi\Vert_{L^2(M_{\mathrm{ext}}\bigcap\bigsqcup\limits_{i=1}^nB_{3t}(p_i))}&\le C(t^{\frac{3}{2}}+\overline m^{-1})\Vert f\Vert_{\mathcal H^0_\beta} .
    \end{align*}
\end{proposition}
\begin{proof}
    Extend $f^\parallel$ by zero to $M$, let $\zeta$ and $u$ be the solutions of Propositions~\ref{proposition: longitudinal} and \ref{proposition: transverse} for $f^\parallel$ and $f^\perp$, respectively, and put
    \begin{equation*}
        \xi:=\zeta\oplus d_2^*u\qquad\text{on }M_{\mathrm{ext}}.
    \end{equation*}
    By \eqref{equation: parallel identities} and Lemma~\ref{lemma: Weitzenbock D} with $e_0=0$, we have $D\xi=(f^\parallel\oplus f^\perp,0)=(f,0)$.

    Since $f$ vanishes on $B_{\frac{r_0}{2}}(p_i)$, we have $w\ge\frac{r_0}{2}$ wherever $f\neq0$. Hence the two propositions give
    \begin{equation*}
        \Vert\zeta\Vert_{L^2_{1,\beta}}+\overline m\Vert e^{-\frac12\beta\rho}d_2^*u\Vert_{L^2(M_{\mathrm{ext}})}\le C\Vert f\Vert_{\mathcal H^0_\beta},
    \end{equation*}
    which proves the first estimate. For the second one, note that $D_0\zeta=0$ on $B_{\frac{r_0}{2}}(p_i)$. Interior elliptic estimates give $\lvert\zeta\rvert\le C\Vert f\Vert_{\mathcal H^0_\beta}$ on $B_{\frac{r_0}{4}}(p_i)\supset B_{3t}(p_i)$, and hence $\Vert e^{-\frac12\beta\rho}\zeta\Vert_{L^2(B_{3t}(p_i))}\le Ct^{\frac{3}{2}}\Vert f\Vert_{\mathcal H^0_\beta}$. Together with the bound for $d_2^*u$, this proves the second estimate.
\end{proof}

\subsubsection{Patching}\label{subsubsection: Patching}

We now patch the local solutions together. For $i=1,\cdots,n$, let $\gamma_i:M\to[0,1]$ be a smooth function, independent of $\overline m$ and $t$, such that $\gamma_i=1$ on $B_{\frac{r_0}{2}}(p_i)$ and $\supp\gamma_i\subset B_{r_0}(p_i)$. Let $\gamma_0:M\to[0,1]$ be a smooth function such that
\begin{equation*}
    \gamma_0=0\text{ on }\bigsqcup_{i=1}^nB_{2t}(p_i),\qquad\gamma_0=1\text{ outside }\bigsqcup_{i=1}^nB_{3t}(p_i),\qquad\lvert\nabla\gamma_0\rvert\le\frac2t .
\end{equation*}
If $\overline m\ge t^{-1}$, then $w\le4t$ on $B_{3t}(p_i)$, so
\begin{equation}\label{equation: cutoff}
    w\lvert\nabla\gamma_i\rvert\le C\qquad(i=0,1,\cdots,n),
\end{equation}
where $C$ is independent of $t$.

\begin{proof}[Proof of Theorem~\ref{theorem: surjective}]
    Let $f\in\mathcal H^0_\beta$. We follow the three steps above.

    First, for $i=1,\cdots,n$, let $f_i:=f$ on $B_{\frac{r_0}{2}}(p_i)$ and $f_i:=0$ on $B_{r_0}(p_i)\setminus B_{\frac{r_0}{2}}(p_i)$, and let $\xi^{(i)}$ be the pair of Proposition~\ref{proposition: ball problem} for $f_i$.

    Second, extending each $\gamma_i\xi^{(i)}$ by zero, we put
    \begin{equation*}
        f_0:=f-\sum_{i=1}^nd_2(\gamma_i\xi^{(i)}).
    \end{equation*}
    Since $\gamma_i=1$ on $B_{\frac{r_0}{2}}(p_i)$, $f_0$ vanishes on $\bigsqcup\limits_{i=1}^nB_{\frac{r_0}{2}}(p_i)$; by \eqref{equation: cutoff}, \eqref{equation: commutator} and Proposition~\ref{proposition: ball problem}, $\Vert f_0\Vert\le C\Vert f\Vert$. Let $\xi^{(0)}$ be the pair of Proposition~\ref{proposition: exterior problem} for $f_0$.

    Third, we define
    \begin{align*}
        \mathcal F:\mathcal H^0_\beta&\to\mathcal H^1_\beta\\
        f&\mapsto\sum_{i=0}^n\gamma_i\xi^{(i)},
    \end{align*}
    where each term is extended by zero. By Lemma~\ref{lemma: interior estimate} for each $\gamma_i\xi^{(i)}$, \eqref{equation: cutoff}, \eqref{equation: commutator} and Propositions~\ref{proposition: ball problem} and \ref{proposition: exterior problem}, we have $\Vert\mathcal Ff\Vert_{\mathcal H^1_\beta}\le C\Vert f\Vert_{\mathcal H^0_\beta}$. Since $\gamma_0=1$ on $\supp f_0$, we have $d_2(\gamma_0\xi^{(0)})=f_0+h_0$ with $h_0:=d_2(\gamma_0\xi^{(0)})-\gamma_0d_2\xi^{(0)}$, and hence
    \begin{equation*}
        d_2\mathcal Ff=(f-f_0)+(f_0+h_0)=f+h_0.
    \end{equation*}
    Since $h_0$ is supported in $\bigsqcup\limits_{i=1}^nB_{3t}(p_i)$, \eqref{equation: cutoff}, \eqref{equation: commutator} and Proposition~\ref{proposition: exterior problem} give
    \begin{align*}
        \Vert d_2\mathcal Ff-f\Vert_{\mathcal H^0_\beta}=\Vert h_0\Vert_{\mathcal H^0_\beta}
        &\le C\Vert e^{-\frac12\beta\rho}\xi^{(0)}\Vert_{L^2(M_{\mathrm{ext}}\bigcap\bigsqcup\limits_{i=1}^nB_{3t}(p_i))}\\
        &\le C_0(t^{\frac{3}{2}}+\overline m^{-1})\Vert f\Vert_{\mathcal H^0_\beta},
    \end{align*}
    where $C_0$ is independent of $t$. We now fix $t$ such that $C_0t^{\frac{3}{2}}\le\frac14$. If $\overline m$ is sufficiently large, then the operator norm of $d_2\mathcal F-\mathrm{Id}$ on $\mathcal H^0_\beta$ is at most $\frac12$. Hence the Neumann series shows that $d_2\mathcal F$ is invertible on $\mathcal H^0_\beta$, and that the operator norm of its inverse is at most $2$. Therefore $\mathcal G:=\mathcal F\circ(d_2\mathcal F)^{-1}$ satisfies $d_2\mathcal G=\mathrm{Id}$ and $\Vert\mathcal Gf\Vert_{\mathcal H^1_\beta}\le C_G\Vert f\Vert_{\mathcal H^0_\beta}$, where $C_G$ is independent of $\overline m$.
\end{proof}

\subsection{Quadratic Terms and the Contraction Mapping Theorem}\label{subsection: Quadratic Terms and the Contraction Mapping Theorem}

Let $\mathcal G$ be the right inverse from Theorem~\ref{theorem: surjective}, and write $(a,\varphi)=\mathcal Gf$. The identity $d_2\mathcal G=\mathrm{Id}$ reduces the perturbation equation to
\begin{equation}
     d_2(\mathcal Gf)+Q(\mathcal Gf,\mathcal Gf)=-e_0
    \iff  f+Q(\mathcal Gf,\mathcal Gf)=-e_0 .\label{equation:f+q(f)=e}
\end{equation}
Thus it suffices to solve this fixed-point equation.

\begin{theorem}[{\cite{MR1079726}}]\label{theorem: contraction mapping theorem}
    Let $B$ be a Banach space. Let $q:B\to B$ be a map with $q(0)=0$ satisfying
    \begin{equation*}
        \Vert q(u)-q(v)\Vert_B\le c_q(\Vert u\Vert_B+\Vert v\Vert_B)\Vert u-v\Vert_B,
    \end{equation*}
    for all $u,v\in B$, where $c_q>0$ is a constant independent of $u$ and $v$. Then for all $e\in B$ satisfying $\Vert e\Vert_B<\frac1{10c_q}$, there is a unique solution $f$ to the fixed-point equation $f+q(f)=e$ with $\Vert f\Vert_B\le2\Vert e\Vert_B$.
\end{theorem}

Recall that the quadratic operator $Q$ of Subsection~\ref{subsection: Deformation Complexes} is symmetric and bilinear. Pointwise,
$\lvert Q(\xi,\xi')\rvert\le C\lvert\xi\rvert\lvert\xi'\rvert$.

\begin{lemma}\label{lemma: quadratic estimate}
    If $\overline m$ is sufficiently large, then for all $\xi,\xi'\in\mathcal H^1_\beta$,
    \begin{equation*}
        \Vert Q(\xi,\xi')\Vert_{\mathcal H^0_\beta}\le C\overline m^{\frac{1}{2}}\Vert\xi\Vert_{\mathcal H^1_\beta}\Vert\xi'\Vert_{\mathcal H^1_\beta} .
    \end{equation*}
\end{lemma}
\begin{proof}
    This follows from the pointwise bound for $Q$ and Lemma~\ref{lemma: weighted Sobolev}.
\end{proof}

The error $e_0$ is supported in $\bigsqcup\limits_{i=1}^nB_{2\varepsilon_i}(p_i)$. On each such ball, $\lvert e_0\rvert\le C$ and $w\le3\varepsilon_i$, so
\begin{equation}\label{equation: e0 norm}
    \Vert e_0\Vert^2_{\mathcal H^0_\beta}\le C\sum_{i=1}^n\int_{B_{2\varepsilon_i}(p_i)}w^2\vol_g\le C\sum_{i=1}^n\varepsilon_i^5,\qquad\text{i.e.}\qquad\Vert e_0\Vert_{\mathcal H^0_\beta}\le C\overline m^{-5/4}.
\end{equation}

\begin{theorem}\label{theorem: existence of solution}
    If $\overline m$ is sufficiently large, then there is $f\in\mathcal H^0_\beta$ with $f+Q(\mathcal Gf,\mathcal Gf)=-e_0$ and $\Vert f\Vert_{\mathcal H^0_\beta}\le2\Vert e_0\Vert_{\mathcal H^0_\beta}$. Moreover, $(a,\varphi)=\mathcal Gf$ satisfies
    \begin{equation*}
        d_2(a,\varphi)+Q((a,\varphi),(a,\varphi))+e_0=0,\qquad\Vert(a,\varphi)\Vert_{\mathcal H^1_\beta}\le C\overline m^{-5/4}.
    \end{equation*}
\end{theorem}
\begin{proof}
    For fixed $\overline m$, we apply Theorem~\ref{theorem: contraction mapping theorem} to the Banach space $B:=\mathcal H^0_\beta$ and the map $q(f):=Q(\mathcal Gf,\mathcal Gf)$. Since $Q$ is symmetric and bilinear, we have
    $q(f)-q(f')=Q(\mathcal G(f+f'),\mathcal G(f-f'))$.
    Theorem~\ref{theorem: surjective} and Lemma~\ref{lemma: quadratic estimate} then imply that, for all $f,f'\in B$,
    \begin{equation*}
        \Vert q(f)-q(f')\Vert_B\le c_q(\Vert f\Vert_B+\Vert f'\Vert_B)\Vert f-f'\Vert_B,\qquad c_q:=C_QC_G^2\overline m^{\frac{1}{2}},
    \end{equation*}
    where $C_Q$ is the constant of Lemma~\ref{lemma: quadratic estimate}. Let $C_e$ be the constant in \eqref{equation: e0 norm}. Since $C_Q$, $C_e$ and $C_G$ are independent of $\overline m$, we have
    \begin{equation*}
        10\,c_q\,\Vert e_0\Vert_B\le10C_QC_eC_G^2\,\overline m^{-\frac34},
    \end{equation*}
    which is less than $1$ if $\overline m$ is sufficiently large. For such $\overline m$, Theorem~\ref{theorem: contraction mapping theorem} with $e=-e_0$ gives a solution $f\in B$ with $\Vert f\Vert_B\le2\Vert e_0\Vert_B$. The pair $(a,\varphi)=\mathcal Gf$ then solves \eqref{equation: equation for the perturbation}, and
    \begin{equation*}
        \Vert(a,\varphi)\Vert_{\mathcal H^1_\beta}\le C_G\Vert f\Vert_B\le2C_G\Vert e_0\Vert_B\le2C_eC_G\,\overline m^{-5/4}.\qedhere
    \end{equation*}
\end{proof}

\subsection{Proof of Theorem~\ref{theorem: main theorem}}\label{subsection: Proof of the main theorem}

Let $(A,\Phi)=(A_{\mathrm{app}}+a,\Phi_{\mathrm{app}}+\varphi)$ be the solution from Theorem~\ref{theorem: existence of solution}. We prove that it is gauge equivalent to a smooth irreducible monopole of charge $n$ and mass $m=\overline m$. On the end, $(A_{\mathrm{app}},\Phi_{\mathrm{app}})$ is the scaled $\SU(2)$-Dirac monopole with $\Phi_{\mathrm{app}}=\phi_D^{\overline m}\hat\sigma$, so Theorem~\ref{proposition:existence of harmonic function PhiD} with $k_1^O=n$ and $m_1^O=1$ gives
\begin{equation}\label{equation: lambda and m}
\lvert\Phi_{\mathrm{app}}\rvert=\phi_D^{\overline m}=\overline m+\frac{4\pi n}{\vol(\Sigma)}\rho+O(e^{-\delta'\rho}), \qquad\int_{\{\rho=R\}}*d\lvert\Phi_{\mathrm{app}}\rvert=4\pi n.
\end{equation}
The harmonic expansion and interior gradient estimates also give $\lvert d\phi_D^{\overline m}\rvert\le C$ on the end, with $C$ independent of $\overline m$. We first establish regularity and end asymptotics for $(A,\Phi)$.

\begin{lemma}\label{lemma: regularity and decay}
    \begin{enumerate}[label=\textbf{(\arabic*)},listparindent=\parindent]
        \item There is a gauge transformation $\mathsf g$ of class $L^2_{2,loc}$ such that $(A',\Phi'):=\mathsf g^*(A,\Phi)$ is smooth and satisfies the Bogomolny equation~\eqref{equation: Bogomolny equation}.
        \item For every $\beta'\in(\frac12\beta,0)$, there are $R_1$ and $C'$, depending on $\overline m$, such that, on $\{\rho\ge R_1\}$, $\Phi$ is continuous and nonvanishing, and $\hat\Phi:=\frac{\Phi}{\lvert\Phi\rvert}$ satisfies
        \begin{equation*}
            \lvert\Phi-\Phi_{\mathrm{app}}\rvert+\big\lvert d\lvert\Phi\rvert-d\lvert\Phi_{\mathrm{app}}\rvert\big\rvert+\lvert d_A\hat\Phi\rvert\le C'e^{\beta'\rho}.
        \end{equation*}
    \end{enumerate}
\end{lemma}
\begin{proof}
    We first prove the local regularity statement. Let $U$ be a small geodesic ball, and trivialize $P$ over $U$. Since $a,\varphi\in L^2_1(U)\subset L^6(U)$, the curvature
    $F_A=F_{A_{\mathrm{app}}}+d_{A_{\mathrm{app}}}a+\frac12[a\wedge a]$
    belongs to $L^2(U)$, and hence its $L^{\frac32}$-norm, which is invariant under rescaling, tends to zero as $U$ shrinks. Since $\frac{3}{2}\le2<3$, Uhlenbeck's gauge fixing theorem (\cite{Uhlenbeck_1982}), applied after rescaling $U$ to unit size, gives, on a smaller ball, a gauge transformation $\mathsf g_U\in L^2_2(U;\SU(2))$ for which $\tilde A=\mathsf g_U^*A$ satisfies $d^*\tilde A=0$. Writing $\tilde\Phi=\mathsf g_U^{-1}\Phi\mathsf g_U$, the Bogomolny equation becomes
    $D_0(\tilde A,\tilde\Phi)=(-\frac12*[\tilde A\wedge\tilde A]+[\tilde A,\tilde\Phi],0)$,
    where $D_0$ is the operator in \eqref{equation: D0}, now acting on $\su(2)$-valued forms. Since $(\tilde A,\tilde\Phi)\in L^2_1\subset L^6$, the right hand side lies in $L^3$, so interior elliptic estimates for $D_0$ give $(\tilde A,\tilde\Phi)\in L^3_1$ on a smaller ball. Then $(\tilde A,\tilde\Phi)\in L^q$ for every $q<\infty$, the right hand side lies in $L^q$, and hence $(\tilde A,\tilde\Phi)\in L^q_1\subset C^0$ for $q>3$. Now the right hand side lies in $L^q_1$, so $(\tilde A,\tilde\Phi)\in L^q_2$, and by induction $\tilde A$ and $\tilde\Phi$ are smooth. On overlaps, the transition functions $h=\mathsf g_U^{-1}\mathsf g_{U'}$ satisfy $dh=h\tilde A'-\tilde Ah$, so they are smooth as well. They define a smooth bundle $\widetilde P$ on which the transformed pair is smooth. The local Coulomb gauges give an $L^2_{2,loc}$ isomorphism $\widetilde P\to P$. Since every principal $\SU(2)$-bundle over a $3$-manifold is trivial, we may identify $\widetilde P$ smoothly with $P$; composing the two identifications gives the required $L^2_{2,loc}$ gauge transformation.

    We next prove (2). In the construction $\mathcal G=\mathcal F(d_2\mathcal F)^{-1}$, the cut-off functions satisfy $\gamma_0=1$ and $\gamma_i=0$ outside the balls $B_{r_0}(p_i)$. Thus $(a,\varphi)$ agrees there with an exterior solution $\xi^{(0)}$ as in the proof of Theorem~\ref{theorem: surjective}. Hence $d_1^*(a,\varphi)=0$, and, since $e_0=0$, the perturbation equation becomes
    $D(a,\varphi)=(-Q(a,\varphi),0)$.
    Elliptic regularity shows that $(a,\varphi)$ is smooth away from the balls.

    To obtain a quantitative estimate, let $x_0$ be a point with $\rho(x_0)=R$, where $R\geq\overline m$ is large, and let $U$ be the geodesic ball of radius $2$ centered at $x_0$. Then $w\ge1$ on $U$, and the weighted bound on $(a,\varphi)$ gives $\Vert(a,\varphi)\Vert_{L^2(U)}+\Vert\nabla_{A_{\mathrm{app}}}(a,\varphi)\Vert_{L^2(U)}\le Ce^{\frac12\beta R}$. Since $\hat\sigma$ is parallel, we may choose a trivialization of $P$ over $U$ in which $\hat\sigma=\frac{\sigma_3}{2}$ and $A_{\mathrm{app}}=a_0\frac{\sigma_3}{2}$, where $a_0$ is the radial primitive of $*d\phi_D^{\overline m}$ in normal coordinates $x$ centered at $x_0$. Since all the derivatives of $\phi_D^{\overline m}$ of positive order are bounded on the end, $a_0$ is bounded in $C^k$ for every $k$, and $\Phi_{\mathrm{app}}=\phi_D^{\overline m}\frac{\sigma_3}{2}$ satisfies $\lvert\Phi_{\mathrm{app}}\rvert\le C(\overline m+R)\le CR$, while its derivatives of positive order are bounded. In the coordinates $y=Rx$, the operator $R^{-1}D$ is a first-order elliptic operator on $\{\lvert y\rvert\le2\}$, whose principal part comes from the metric $g$, and whose zeroth-order part is $R^{-1}$ times that of $D$, which involves $g$, $a_0$ and $\ad(\Phi_{\mathrm{app}})$. Hence its coefficients are bounded in $C^k$ in the $y$-coordinates, and it is uniformly elliptic, independently of $R$. The pair $(a,\varphi)$ satisfies $R^{-1}D(a,\varphi)=(-R^{-1}Q(a,\varphi),0)$, and its $L^2_1$-norm on $\{\lvert y\rvert\le2\}$ is at most $CR^{\frac32}e^{\frac12\beta R}$, which is less than $1$ if $R$ is large. Hence, in each step of the above bootstrapping argument, the quadratic term is bounded by a constant times this norm, and we obtain
    \begin{equation*}
        \sup_{\lvert y\rvert\le1}(\lvert(a,\varphi)\rvert+\lvert\partial_y(a,\varphi)\rvert)\le C\Vert(a,\varphi)\Vert_{L^2_1(\{\lvert y\rvert\le2\})}\le CR^{\frac32}e^{\frac12\beta R},
    \end{equation*}
    where $C$ is independent of $R$. Since $\partial_x=R\partial_y$, this gives $\lvert(a,\varphi)\rvert+\lvert\nabla_{A_{\mathrm{app}}}(a,\varphi)\rvert\le CR^{\frac52}e^{\frac12\beta R}$ at every point with $\rho=R$, for every large $R$.

    We now deduce the estimates for $\lvert\Phi\rvert$ and $\hat\Phi$. Since $\langle[a,X],X\rangle=0$ for every $X\in\mathfrak g_P$, and $\nabla_{A_{\mathrm{app}}}\hat\sigma=0$, we have
    \begin{gather*}
        d\lvert\Phi\rvert=\langle\nabla_{A_{\mathrm{app}}}\Phi,\hat\Phi\rangle,\qquad
        d_A\hat\Phi=\frac{(\nabla_{A_{\mathrm{app}}}\Phi)^\perp}{\lvert\Phi\rvert}+[a,\hat\Phi],\\
        \nabla_{A_{\mathrm{app}}}\Phi=d\phi_D^{\overline m}\hat\sigma+\nabla_{A_{\mathrm{app}}}\varphi,
    \end{gather*}
    where $\perp$ denotes projection orthogonal to $\hat\Phi$. Now $\Phi-\Phi_{\mathrm{app}}=\varphi$, and hence
    $\lvert\hat\Phi-\hat\sigma\rvert\le\frac{2\lvert\varphi\rvert}{\lvert\Phi_{\mathrm{app}}\rvert}$.
    By \eqref{equation: lambda and m}, $\lvert\Phi_{\mathrm{app}}\rvert\ge\frac{2\pi n}{\vol(\Sigma)}\rho$ on the end, while $\lvert d\phi_D^{\overline m}\rvert\le C$. Combining these facts with the $C^1$-estimate above gives, on $\{\rho=R\}$,
    \begin{equation*}
        \lvert\Phi-\Phi_{\mathrm{app}}\rvert+\big\lvert d\lvert\Phi\rvert-d\lvert\Phi_{\mathrm{app}}\rvert\big\rvert+\lvert d_A\hat\Phi\rvert
        \le C(\lvert(a,\varphi)\rvert+\lvert\nabla_{A_{\mathrm{app}}}(a,\varphi)\rvert)
        \le CR^{\frac52}e^{\frac12\beta R}.
    \end{equation*}
    Given any $\beta'>\frac12\beta$, the polynomial factor can be absorbed into the exponential, so $R^{\frac52}e^{\frac12\beta R}\le C'e^{\beta'R}$. This proves the asserted decay. It also shows that $\Phi$ is nonvanishing sufficiently far out on the end.
\end{proof}

\paragraph{Mass and charge.}
Put $m:=\overline m$ and $\hat\Phi':=\frac{\Phi'}{\lvert\Phi'\rvert}$, and fix $\beta'\in(\max\{\frac12\beta,-\delta'\},0)$, where $\delta'$ is as in \eqref{equation: lambda and m}. Combining \eqref{equation: lambda and m} with Lemma~\ref{lemma: regularity and decay}, and using $\lvert d\phi_D^{\overline m}\rvert\le C$, we obtain on the end
\begin{equation}\label{equation: asymptotics of solution}
    \begin{aligned}
        &\lvert\Phi'\rvert=\lvert\Phi\rvert=m+\frac{4\pi n}{\vol(\Sigma)}\rho+O(e^{\beta'\rho}),\qquad\lvert d\lvert\Phi'\rvert\rvert\le C,\\
        &\int_{\{\rho=R\}}*d\lvert\Phi'\rvert=4\pi n+O(e^{\beta'R}),\qquad\lvert d_{A'}\hat\Phi'\rvert=\lvert d_A\hat\Phi\rvert\le C'e^{\beta'\rho}.
    \end{aligned}
\end{equation}
Thus $(A',\Phi')$ has mass $m$ in the sense of Definition~\ref{definition: mass and charge ACyl}. By Lemma~\ref{lemma: regularity and decay}(2) and \eqref{equation: lambda and m}, we have $\lvert\Phi-\Phi_{\mathrm{app}}\rvert<\lvert\Phi_{\mathrm{app}}\rvert$ on $\{\rho=R\}$ for large $R$, so $\Phi_{\mathrm{app}}+s(\Phi-\Phi_{\mathrm{app}})$, $s\in[0,1]$, is a homotopy through nonvanishing sections. Hence $\frac{\Phi}{\lvert\Phi\rvert}$ and $\frac{\Phi_{\mathrm{app}}}{\lvert\Phi_{\mathrm{app}}\rvert}=\hat\sigma$ have the same degree on $\{\rho=R\}$, computed in a global trivialization of $P$, which exists since every principal $\SU(2)$-bundle over $M$ is trivial. This degree is $n$. Indeed, $\hat\sigma$ is defined on $M\setminus\{p_1,\cdots,p_n\}$, and $\{\rho=R\}$ is homologous there to the union of small spheres around $p_1,\cdots,p_n$. On $B_{\varepsilon_0}(p_i)$, the global trivialization differs from the one given by the framing $\eta_i$, which extends over $p_i$ (Lemma~\ref{lemma:extending the principal SU(2)-bundle P to M-...}), by a map to $\SU(2)$, which does not change the degree on a small sphere around $p_i$. In the latter trivialization, $\hat\sigma=\hat x\cdot\sigma$ has degree $1$ on this sphere. Since $\Phi'=\mathsf g^{-1}\Phi\mathsf g$ with $\mathsf g$ continuous, $(A',\Phi')$ has charge $n$.

\begin{proposition}\label{proposition: irreducible}
    The monopole $(A',\Phi')$ is irreducible, i.e.\ its stabilizer in the gauge group of $P$ is the center $\{\pm1\}$.
\end{proposition}
\begin{proof}
    Suppose that a non-central gauge transformation $\mathsf h$ stabilizes $(A',\Phi')$. Then $\mathsf h$ is parallel and commutes with $\Phi'$. Its eigenvalues $e^{\pm i\vartheta}$, with $\vartheta\in(0,\pi)$, are constant, so we can write $\mathsf h=\exp(\vartheta\tau)$ for a parallel section $\tau$ of constant norm. It follows that $\Phi'=\frac{f\tau}{\lvert\tau\rvert}$ for a real-valued function $f$ with $\lvert f\rvert=\lvert\Phi'\rvert$.

    Since $\tau$ is parallel, the Bogomolny equation and the Bianchi identity give
    $F_{A'}=*df\otimes\frac{\tau}{\lvert\tau\rvert}$ and $d*df=0$.
    By Lemma~\ref{lemma: regularity and decay}(2), $f=\pm\lvert\Phi'\rvert$ on $\{\rho\ge R_1\}$, with a fixed sign. Stokes' theorem and \eqref{equation: asymptotics of solution} then give
    \begin{equation*}
        0=\int_{\{\rho\le R\}}d*df=\int_{\{\rho=R\}}*df
        =\pm(4\pi n+O(e^{\beta'R})),
    \end{equation*}
    which is impossible for large $R$, since $n\ne0$.
\end{proof}

\begin{proof}[Proof of Theorem~\ref{theorem: main theorem}]
    Here $m=\overline m$. By Theorem~\ref{theorem: existence of solution}, there is a perturbation $(a,\varphi)\in\mathcal H^1_\beta$ satisfying \eqref{equation: equation for the perturbation} and
    $\Vert(a,\varphi)\Vert_{\mathcal H^1_\beta}\le C\overline m^{-5/4}=Cm^{-5/4}$.
    Lemma~\ref{lemma: regularity and decay} gives a gauge equivalent smooth monopole $(A',\Phi')$. The preceding argument shows that it has mass $m$ and charge $n$, while Proposition~\ref{proposition: irreducible} shows that it is irreducible.
\end{proof}

\bibliographystyle{alpha}
\bibliography{ref}

\end{document}